\documentclass[11pt,reqno]{amsart}
\usepackage{amssymb}
\usepackage{cite}
\usepackage{graphicx}
\newcommand{\vdiv}{\mathop{\rm div}}
\newcommand{\RRM}{\mathbb{R}}

\newcommand{\ZZM}{\mathbb{Z}}

\newcommand{\NNM}{\mathbb{N}}
\newcommand{\SSM}{\mathbb{S}}

\newcommand{\bea}{\begin{eqnarray*}}
\newcommand{\eea}{\end{eqnarray*}}
\newcommand{\be}{\begin{equation}}
\newcommand{\ee}{\end{equation}}

\newcommand{\cla}{\mathcal{A}}
\newcommand{\clb}{\mathcal{B}}
\newcommand{\cll}{\mathcal{L}}
\newcommand{\cls}{\mathcal{S}}
\newcommand{\clu}{\mathcal{U}}
\newcommand{\clf}{\mathcal{F}}

\newcommand{\clv}{\mathcal{V}}
\newcommand{\clm}{\mathcal{M}}

\newcommand{\tlx}{{\tilde x}}
\newcommand{\tly}{{\tilde y}}
\newcommand{\0}{\Omega}
\newcommand{\G}{\Gamma}
\newcommand{\p}{\partial}
\newcommand{\wh}{\widehat}
\newcommand{\ov}{\overline}
\newcommand{\wt}{\widetilde}
\newcommand{\tlh}{{\tilde h}}
\newcommand{\tlu}{{\tilde u}}
\newcommand{\tlt}{{\tilde t}}

\newcommand{\eps}{\varepsilon}

\newtheorem{lemma}{Lemma}
\newtheorem{theorem}[lemma]{Theorem}

\newtheorem{prop}[lemma]{Proposition}
\theoremstyle{remark}

\makeatletter
\@addtoreset{equation}{section}
\@addtoreset{lemma}{section}
\catcode`@=12

\begin{document}
\title
[{Hele-Shaw flow in thin threads: A rigorous limit result}] {Hele-Shaw flow in
thin threads:\\ A rigorous limit result}

\author{Bogdan-Vasile Matioc\ \ \ \ and \ \ \ \ Georg Prokert}
\address{\hspace{-2.5ex}B.-V. Matioc,
Institute of Applied Mathematics, Leibniz University Hannover,
 Germany.\newline
 E-mail: {\tt matioc@ifam.uni-hannover.de}\vspace{3mm}
\newline
G. Prokert, Faculty of Mathematics and Computing Science, Technical
University Eindhoven, 
\newline The Netherlands.
 E-mail: {\tt g.prokert@tue.nl}}

\begin{abstract}
We rigorously prove the convergence of appropriately scaled solutions of the
2D Hele-Shaw moving boundary problem with surface tension in the limit of
thin threads to the solution of the formally corresponding Thin Film
equation. The proof is based on scaled parabolic estimates for the
nonlocal, nonlinear evolution equations that arise from these problems. 

\bigskip


\end{abstract}
\maketitle
\section{Introduction and main result}
In  theoretical fluid mechanics, the investigation of limit cases in which
the thickness of the flow domain is small compared to its other
lengthscale(s) is a classical subject. In most situations, the simplified
equations describing these limit cases are derived formally from the original
problem by expansion with respect to the small parameter describing the ratio
of the lengthscales. Although lots of work have been devoted to the
investigation of the limit equations (lubrication equations or various
so-called Thin Film equations), the question of justifying the approximation
by comparing the solutions of the original problem to those of the
corresponding limit problem is less studied. This is true in particular when
a moving boundary is an essential part of the original problem. 

In the case of 2D Hele-Shaw flow in a thin layer, the only  rigorous
limit result known to us has been proved by Giacomelli and Otto \cite{giot}.
Their approach is based on variational methods and can even handle degenerate
cases and complicated geometries. However, the existence of global smooth
solutions of the Hele-Shaw problem under consideration has to be presupposed,
and the obtained result on the closeness to some solution of the Thin Film
equation is in a relatively weak sense and technically rather complicated.

It is the aim of the present paper to provide a justification of the same
limit equation using quite different, more standard methods.
Starting from a strictly positive solution of the Thin Film equation, we prove the solvability of the corresponding moving boundary problems for large
times. If the initial shape is smooth, approximations to arbitrary
order and in arbitrarily strong norms are obtained. Moreover, our approach is
straightforwardly generalizable to a multidimensional setting.  However, it is
restricted to the nondegenerate case of strictly positive film thickness and to
a simple layer geometry.

More precisely, we consider 2D Hele-Shaw flow in a periodic, thin liquid domain
(i.e. a thread), symmetric about the $x$-axis, with surface tension as sole
driving mechanism for the flow (see Fig. \ref{thread}).
\begin{figure}\label{thread}
\begin{center}
 \includegraphics[width=0.7\textwidth]{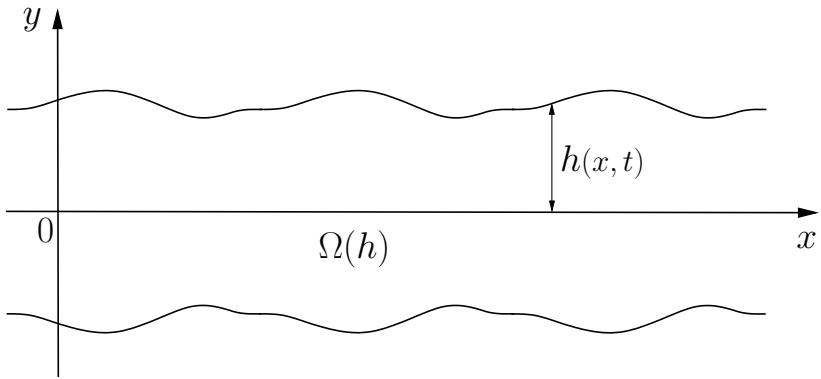}
\end{center}
\caption{ The considered geometry and basic notation}
\end{figure}

This problem (in the unscaled version, with surface tension coefficient
normalized to $1$) consists in finding a positive function
$h\in C^1([0,T],C^2(\SSM))$, $\SSM=\RRM/[0,2\pi)$, and, for each $t\in[0,T]$,
a function $u$ defined on 
\[\Omega(h):=\{(x,y)\in\SSM\times\RRM\,|\,|y|<h(x)\}\]
such that
\begin{equation}\label{mbp0}
\left\{\begin{array}{rcll}
-\Delta u&=&0&\mbox{ in $\Omega(h)$},\\
u&=&-\kappa(h)&\mbox{ on $\p\Omega(h)$},\\
\partial_t h+\nabla u\cdot(-h',\,1)^\top&=&0&\mbox{ on
$\G_+(h)$},
\end{array}\right.
\end{equation}
where $\G_+(h):=\{(x,h(x))\,|\, x\in\SSM\},$ and
\[\kappa(h):=\frac{h''}{(1+(h')^2)^{3/2}}\]
is the curvature of $\G_+(h)$. 
(Here and in the sequel, for the sake
of brevity we write $h$ instead of $h(t)$ or $h(\cdot,t)$ when no confusion
seems likely. Moreover, we identify functions  on $\SSM$ with
functions on $\p\0$  using the pull-back along $x\mapsto(x,\pm h(x))$.) Note
that $u$ represents the normalized pressure in the
Hele-Shaw cell, and, in view of Darcy's law, the first and third equation are
the incompressibility condition and the usual kinematic free boundary
condition. 
Moreover, a uniqueness argument shows that $u$ is symmetric 
with respect to $y$, i.e. $u(x,y)=u(x,-y)$ for all $(x,y)\in \Omega(h),$ 
meaning that $u_y=0$ on $\SSM\times\{0\}.$
Thus, this setting corresponds to the case when the bottom of the Hele-Shaw
cell,  which we take  as being   $\SSM\times\{0\},$
is impermeable.
We refrain from discussing further modelling aspects and refer
instead to the extensive literature on the subject, see e.g. \cite{eloc},
Ch.\ 1. Well-posedness results for \eqref{mbp0} (in slightly different
geometric settings and various classes of functions) have been proved in e.g.
\cite{bafr, essi, prhs} and in \cite{EM1, EM2} for non-Newtonian fluids.  

To consider thin threads we introduce a scaling parameter $\eps$, $0<\eps\ll
1$, and rescale by
\[x=\tlx,\quad y=\eps\tly, \quad h=\eps\tlh,\quad
\tlu(\tlx,\tly)=u(\tlx,\eps\tly).\]
Then $\tlu$ is defined on $\Omega(\tlh)$, and $(\tlh,\tlu)$ satisfies the
$\eps$-dependent problem
\begin{equation}\label{mbp1}
\left\{\begin{array}{rcll}
-\eps^2\tlu_{\tlx\tlx}-\tlu_{\tly\tly} &=&0&\mbox{ in $\Omega(\tlh)$},\\
\tlu&=&-\eps\kappa(\eps,\tlh)&\mbox{ on $\p\Omega(\tlh)$},\\
\partial_t \tlh+\nabla_{(\tlx,\tly)}\tlu\cdot(-\tlh',\,\eps^{-2})^\top
&=&0&\mbox{ on $\Gamma_+(\tlh)$},
\end{array}\right.
\end{equation}
where
\[\kappa(\eps,\tlh):=\frac{\tlh''}{(1+\eps^2(\tlh')^2)^{3/2}}.\]
Expanding $\tlu$ and $\kappa_\eps$ formally in power series in $\eps$ we
obtain
\[\tlu=\eps u_1+\eps^3u_3+O(\eps^4),\]
where in particular
\[(u_1)_\tlx=-\tlh''',\quad (u_3)_\tly=\tlh\tlh''''
\qquad\mbox{on $\G_+(\tlh)$}.\]
So
\[\partial_t\tlh+\eps(\tlh\tlh''')'=O(\eps^2),\]
and after rescaling $t=\eps^{-1}\tlt$, suppressing tildes and neglecting
higher order terms we obtain the well-known Thin Film equation
\be\label{thf}
\partial_th_0+(h_0h_0''')'=0\quad\text{on $\SSM$}.
\ee
Observe that \eqref{thf} is a fourth order parabolic equation for positive $h_0$
which degenerates as $h_0$ approaches $0$. For a review of the extensive
literature on this and related equations we refer to \cite{huls}. In the
modelling context discussed here, \eqref{thf} is used in \cite{const, dupont} to
study the breakup behavior of Hele-Shaw threads. Note that \eqref{thf} and
its multidimensional analogue
\begin{equation}\label{eq:tfea}
\partial_th+\vdiv(h\nabla\Delta h)=0,
\end{equation}
 also appear in models of ground water flow \cite{huls}. 

In view of the time rescaling, $h_0$ should be an approximation to 
\[h_\eps:=[t\mapsto\tlh(\eps^{-1}t)],\]
where $\tlh$ solves \eqref{mbp1}, i.e. $h_\eps$ solves (with an appropriate
$v$ and omitting tildes)
\begin{equation}\label{mbp2}
\left\{
\begin{array}{rcll}
-\eps^2v_{xx}-v_{yy} &=&0&\mbox{ in $\Omega(h_\eps)$},\\
v&=&-\kappa(\eps,h_\eps)&\mbox{ on $\p\Omega(h_\eps)$},\\
\partial_t h_\eps+\nabla v\cdot(-h_\eps',\,\eps^{-2})^\top&=&0&\mbox{ on $\G_+(h_\eps)$.}
\end{array}\right.
\end{equation}
(Of course, \eqref{mbp2} can be obtained immediately from \eqref{mbp0} by
choosing the ``correct'' scaling $u=\eps \tlu$, $t=\eps^{-1}\tlt$ at once.
That scaling, however, is in itself rather a result of the above
calculations.)

Our rigorous justification of the Thin Film approximation for \eqref{mbp0} will
therefore consist in showing that for any positive initial datum $h^\ast$
(from a suitable class of functions), the common initial condition
\[h_0(0)=h_\eps(0)=h^\ast\]
implies existence and uniqueness  of solutions to \eqref{thf} and
\eqref{mbp2} (for all sufficiently small $\eps$) on the same time interval,
and 
\be\label{conv}
h_\eps\rightarrow h_0\quad\mbox{as $\eps\downarrow 0$}
\ee
(in a suitable sense). 
This will be made precise in our main result Theorem \ref{thm2} below.
Observe that this, in particular, implies that the existence time in the
original timescale, i.e. for solutions of \eqref{mbp1}, goes to infinity as
$\eps$ becomes small.

We will make use of the following preliminary, nonuniform well-posedness result
for solutions to \eqref{mbp2}. It is not optimal with respect to the demanded
regularity but this is not our concern here.
\begin{theorem}\label{thm1}
Let $s\geq7$ be an  integer  and   $h^\ast\in H^{s+1}(\SSM)$  a positive
function.
Then we have:
 
\begin{itemize}
\item[\rm (i)] Existence and uniqueness: 
Problem \eqref{mbp2} with initial condition $h_\eps(0)=h^\ast$ has a unique
maximal solution
\[h_\eps\in C([0,T_\eps), H^{s}(\SSM))\cap C^1([0,T_\eps), H^{s-3}(\SSM))\]
for some $T_\eps=T_\eps( h^\ast)\in(0,\infty]$.

\item[\rm (ii)] Analyticity:
We have
\[[(x,t)\mapsto h_\eps(x,t)]|_{\SSM\times(0,T_\eps)}\in
C^\omega(\SSM\times(0,T_\eps)).\]

\item[\rm (iii)] Blowup:
If $T_\eps<\infty$ then 
\[\mathop{\rm lim \;inf}_{t\uparrow T_\eps}(\min_{x\in\SSM} h_\eps(x,t))=0\quad
\mbox{or}\quad \mathop{\rm lim\;sup}_{t\uparrow T_\eps}\|h_\eps(\cdot,t)\|_{H^s(S)}
=\infty.\]
\end{itemize}
\end{theorem}
Existence and uniqueness of the solution on some time interval $[0,\Theta_\eps(h^\ast)]$ can be proved by parabolic energy estimates and
Galerkin approximations as in \cite{prhs} (for a different geometry).
 The arguments given there can also be used to prove (iii) by a standard continuation argument. 
More precisely, for initial values $h^\ast$ that satisfy $h^\ast>\alpha>0$ and $\|h^\ast\|_s\leq M$ with any positive constants $\alpha$ and $M$, 
the existence time $\Theta_\eps(h^\ast)$ has a positive lower bound depending only on these constants. This implies the blowup result (iii). 
 For a proof
of (ii) we refer to the
framework given in \cite{espr} which is applicable here as well. Essentially,
analyticity follows from the analytic character of all occurring nonlinearities
together with the translational invariance of the problem. 

We are going to state the main result now. It sharpens \eqref{conv} as it also
gives the asymptotics of $h_\eps$ to arbitrary order $n\in\NNM$. This is
achieved by imposing strong smoothness demands on the initial value (or
correspondingly, on $h_0$). To avoid additional technicalities, we have not
strived for optimal regularity results.

\begin{theorem}\label{thm2}
Let $s,n\in\NNM$, $s\geq 10$ be given. 
There is an integer $\beta=\beta(s,n)\in\NNM$
such that for any positive solution 
\[h_0\in C([0,T],H^\beta(\SSM))\cap C^1([0,T],H^{\beta-4}(\SSM))\]
of \eqref{thf} there are $\eps_0=\eps_0(s,n,h_0)>0$,
$C=C(s,n,h_0)$ and functions 
\[h_1,\ldots h_{n-1}\in C([0,T],H^s(\SSM))\]
depending on $h_0$ only such that for all $\eps\in(0,\eps_0)$
\begin{itemize}
\item[\rm (i)] problem \eqref{mbp2} with initial condition
$h_\eps(0)=h_0(0)$ has precisely one solution 
\[h_\eps\in C([0,T],H^{\beta-1}(\SSM))\cap C^1([0,T],H^{\beta-4}(\SSM)),\]
\item[\rm (ii)] 
\[\left\|h_\eps-(h_0+\eps
h_1+\ldots+\eps^{n-1}h_{n-1})\right\|_{C([0,T],H^s(\SSM))}
\leq C\eps^n.\]
\end{itemize}
\end{theorem}
As expected, the functions $h_1,h_2,\ldots$ satisfy the (linear parabolic)
equations arising from formal expansion with respect to $\eps$. 
For details, see Lemma \ref{approxhek} below.

Both this theorem and its proof are in strong analogy to \cite{guprth}, where
the parallel problem for Stokes flow in a thin layer has been discussed. In this
case, the limit equation, also called a Thin Film equation, is 
\begin{equation}
\partial_th+\tfrac{1}{3}\vdiv(h^3\nabla\Delta h)=0.
\end{equation}

The remainder of this paper is devoted to the proof of Theorem \ref{thm2}. For
this purpose, we first transform \eqref{mbp2} to a fixed domain and rewrite the
problem as a nonlinear, nonlocal operator equation for $h_\eps$. Some scaled
estimates are gathered in Section \ref{sec2}. In Section \ref{sec3} we discuss
estimates for our nonlinear operators, while Section \ref{sec4}  provides the
necessary details on the series expansions that are used. Finally, the proof of
Theorem \ref{thm2} is completed in Section \ref{sec5}. 

Technically, the main difficulty in comparison with the unscaled problem is 
the fact that the elliptic estimates for the (scaled) Laplacian and related
operators degenerate as $\eps$ becomes small. To handle this, weighted norms
are introduced, and  estimates in such norms have to be proved. In particular,
the coercivity estimate for the transformed and scaled Dirichlet-Neumann
operator given in Lemma \ref{L:ana} will be pivotal. On the other
hand, the loss of regularity can be compensated by higher order expansions
and interpolation. 
Moreover, as in \cite{guprth}, the ellipticity of the
curvature operator is crucial. Therefore, the corresponding problems 
with gravity instead of surface tension appear intractable by the approach used
here.

\section{Scaled trace inequalities and an extension operator}\label{sec2}
In the remainder of this paper we let $\Gamma_\pm:=\SSM\times\{\pm 1\}$ 
 denote the boundary components of our fixed reference domain
$ \Omega:=\SSM\times(-1,1).$
We write  $H^s(\Omega),$ $H^s(\G_{\pm})$  for the usual 
$L^2-$based Sobolev spaces of order $s\in\RRM,$ while  by definition
$H^s(\p\Omega):= H^s(\Gamma_{+})\times H^s(\Gamma_{-})$.
Functions  $f\in H^s(\G_\pm)  $ may be represented  by their Fourier series
expansions
\[
f(x,\pm 1)=\sum_{p\in\ZZM} \wh f(p) e^{ipx},\qquad x\in\SSM,
\]
with $\wh f(p)$  the $p-$th Fourier coefficient of $f$.
Consequently,  the norm  of $f$ may be defined by the relation
\[
\|f\|_s^{\Gamma_{\pm}}:=\left(\sum_{p\in\ZZM}(1+p^2)^{s}|\wh f(p)|^2\right
)^{1/2}.
\]
Similarly, functions $w\in H^s(\Omega)$ may be written in terms of their Fourier
series
\[
w(x,y)=\sum_{p\in\ZZM}w_p(y)e^{ipx},\qquad (x,y)\in\Omega,
\]
 and the norm of $w$ is given, for $s\geq 0$, by the following expression:
\[
\|w\|_s^\Omega:= \left(\sum_{p\in\ZZM}
\left(\|w_p\|_s^{I}\right)^2+|p|^{2s}\left(\|w_p\|_0^{I}\right)^2\right )^{1/2},
\]
 where $I:=(-1,1).$
Given $\eps\in(0,1)$ and $s\geq1$, we introduce the following scaled norms on $H^s(\Omega)$
\[ 
\|w\|_{s,\eps}^\Omega:=\|w\|_{s-1}^\Omega+\|\p_2
w\|_{s-1}^\Omega+\eps\|w\|_s^\Omega, \qquad w\in H^s(\Omega),
\]
which are equivalent to the standard Sobolev  norm $\|\cdot\|_s^\0$ but,  for
$\eps\to 0$, degenerate to a weaker norm.
The scaling is indicated by the first equation in \eqref{mbp2}, and the  scaled
norms will enable us to take into account the different behaviour of the 
 partial derivatives of the function $v$ in system \eqref{mbp2} with respect  to
$\eps$ when $\eps\to 0$.
To do this we first  introduce an appropriate extension  operator for
functions $f\in H^s(\p\Omega)$ and reconsider some classical estimates in the
weighted norms.  

Given  $\eps\in(0,1),$ we set for simplicity
\[
\nabla _\eps:=(\p_{1,\eps},\p_{2,\eps}):=(\eps\p_1,\p_2).
\]
\begin{lemma}\label{L:Princ} 
\begin{itemize}
\item[$(a)$] There exists a positive constant $C$ such that  for all
$\eps\in(0,1)$ and $w\in H^1(\Omega)$ the following Poincar\'e's inequalities
are satisfied:
\begin{align}
\label{eq:Poin1}
&\|w\|_0^\Omega\leq C \|\nabla_\eps w\|_0^\Omega \quad  \text{if} \ \ \left.
w\right|_{\p\Omega}=0;\\[1ex]
\label{eq:Poin2}
&\|w\|_0^\Omega\leq C \eps^{-1}\|\nabla_\eps w\|_0^\Omega\quad \text{if} \ \ 
\int_{\G_+} w\, d\sigma=0.  
\end{align}
\item[$(b)$] There exists a positive constant $C$ such that the trace inequality
\begin{equation}\label{eq:scest}
\|w\|_t^{\p\Omega}+\sqrt\eps\|w\|^{\p\Omega}_{t+1/2}\leq C\|w\|^\Omega_{t+1,\eps}
\end{equation}
is satisfied by all  $\eps\in(0,1)$, $t\in\NNM$, and $w\in H^{t+1}(\Omega)$.
\item[$(c)$] Given $\eps\in(0,1)$, there exists  linear extension operators 
$E_\pm:H^{t+1/2}(\Gamma_\pm)\longrightarrow H^{t+1}(\Omega)$ such that
$(E_\pm f)|_{\Gamma_\pm}=f$, $E_\pm f$ are even with respect to $y$ and $E_\pm$
satisfy
the estimates
\begin{eqnarray}
\|E_\pm f\|^\Omega_{t+1,\eps}&\leq&
C_t\left(\|f\|_t^{\Gamma_\pm}+\eps^{1/2}\|f\|_{t+1/2}^{\Gamma_\pm}\right),
\label{eq:Esp}\\
\|\partial_2 E_\pm f\|_{-1/2}&\leq&
C\left(\|f\|_{-1/2}^{\Gamma_\pm}+\eps^{1/2}\|f\|_{1/2}^{\Gamma_\pm}\right),
\label{eq:tr-1/2}
\end{eqnarray}
$t\in\NNM$, $f\in H^{t+1/2}(\Gamma_\pm)$, with constants independent of $\eps$. 
\end{itemize}
\end{lemma}
\begin{proof} The proof of $(a)$ is standard while that of  $(b)$ is similar  to
that of \cite[Lemma 3.1]{guprth}.
 To show $(c)$, set for $f(x,1)=\sum_p\wh f(p) e^{ipx}$
\[E_+f(x,y):=\sum_pw_p(y)e^{ipx},\qquad
w_p(y):=y^2e^{i|p|(y^2-1)}\wh f(p).\]
Then $\|w_p\|_{0}^I\leq |\wh f(p)|$ and for $p\neq 0$
\begin{align*}
{\|w_p\|_{0}^I}^{2}= |\wh f(p)|^2  \int_{-1}^1|y^4|e^{2\eps|p|(y^2-1)}\, dy\leq
2|\wh f(p)|^2\int_{0}^1ye^{2\eps|p|(y^2-1)}\, dy\leq \frac{|\wh 
f(p)|^2}{2\eps|p|}.
\end{align*}
Similarly, for $k\in\NNM$,
\[{\|w_p^{(k)}\|_{0}^I}^{2}\leq C_k(1+(\eps|p|)^{2k-1})|\wh f(p)|^2
\leq C_k(1+|p|^{2k-1})|\wh f(p)|^2.\]
Consequently,
\bea
{\|E_+f\|^\Omega_t}^2
&\leq&
\sum_p\left[{\|w_p\|_{t}^I}^2+|p|^{2t}{\|w_p\|_{0}^I}^2\right]
\leq
C{\|f\|^{\Gamma_+}_t}^2,\\
{\|\partial_2 E_+f\|^\Omega_{t+1}}^2
&\leq&
\sum_p\left[{\|w_p'\|_{t}^I}^2+|p|^{2t}{\|w_p'\|_{0}^I}^2\right]
\leq
C\left({\|f\|^{\Gamma_+}_t}^2+\eps{\|f\|^{\Gamma_+}_{t+1/2}}^2\right),\\
\eps^2{\|E_+f\|^\Omega_{t+1}}^2
&\leq&
\eps^2\sum_p\left[{\|w_p\|_{t+1}^I}^2+|p|^{2t+2}{\|w_p\|_{0}^I}^2\right]
\leq
C\eps{\|f\|^{\Gamma_+}_{t+1/2}}^2.
\eea
This proves (\ref{eq:Esp}). To verify (\ref{eq:tr-1/2}), note that
\[{\|\partial_2
E_+f\|^{\Gamma_+}_{-1/2}}^2=\sum_p(1+|p|^2)^{-1/2}|w_p'(1)|^2\]
and 
\[|w_p'(1)|^2\leq C(1+\eps^2|p|^2)|\wh f(p)|^2.\]
This implies the result for $E_+$, the construction for $E_-$ is analogous.
\end{proof}
Using an appropriate smooth
cutoff function, one can construct
a linear extension operator $E:H^{t+1/2}(\p\0)\longrightarrow
H^{t+1}(\Omega)$   such that
$(Ef)|_{\partial\Omega}=f$, $Ef$ is even with respect to $y$ if
$f(\cdot,-1)=f(\cdot,1)$, and $E$ satisfies
the estimates
\begin{eqnarray}
\|Ef\|^\Omega_{t+1,\eps}&\leq&
C_t\left(\|f\|_t^{\p\0}+\eps^{1/2}\|f\|_{t+1/2}^{\p\0}\right),
\label{eq:Esp2}\\
\|\partial_2 Ef\|_{-1/2}^{\p\0}&\leq&
C\left(\|f\|_{-1/2}^{\p\0}+\eps^{1/2}\|f\|_{1/2}^{\p\0}\right),
\label{eq:tr-1/2-2}
\end{eqnarray}
$t\in\NNM$, $f\in H^{t+1/2}(\p\0)$, with constants independent of
$\eps$.

\section{Uniform  estimates for  the scaled and transformed Dirichlet
problem}\label{sec3}

In this section we prove uniform estimates for the solution of the Dirichlet 
problem consisting of the first two equations of \eqref{mbp2}, 
by using the scaled norms defined above.
To this end we first transform the problem \eqref{mbp2} to the strip $\Omega$ by
using
a diffeomorphism depending on the moving boundary $h_\eps$.

Let $\clm$ be $\0$ or $\Gamma_\pm$, $\sigma>\dim \clm/2$, $t\leq\sigma$. We will
repeatedly and without explicit mentioning use the product estimate
\[\|z_1z_2\|_t^\clm\leq C\|z_1\|_t^\clm\|z_2\|_\sigma^\clm,\qquad
z_1\in H^t(\clm),\;z_2\in H^\sigma(\clm).\]

For the remainder of the paper, let $s$ and $s_0$ be such that
$s,s_0+1/2\in\NNM$, $s_0\geq 7/2$, $s\geq 2s_0+3$. (For example, $s_0=7/2$ and
any
$s\geq 10$ is possible, cf. Theorem \ref{thm2}.)
For given $\alpha,M>0$, define the open subset  
$\clu_s:=\clu(s,M,\alpha)$ of $H^s(\SSM)$ by
\[
\clu_s:=\{h\in H^s(\SSM)\,:\, \text{$\|h\|_s<M$ and $\min_{\SSM} h>\alpha$}\}.
\]

Moreover, define the (trivial) maps $\phi_{\pm}:\Gamma_\pm\longrightarrow \SSM$
by
$\phi_\pm(x,\pm 1)=x$.

To avoid losing regularity when transforming the problem onto the fixed
reference manifold $\Omega,$ we modify \cite[Lemma 4.1]{guprth} to obtain the
following result: 
\begin{lemma}[Extension of $h$]\label{L:esth} There exists a map
\[
[h\mapsto\widetilde h]\in\cll(H^\sigma(\SSM), H^{\sigma+1/2}(\Omega)), \quad
\sigma>3/2,
 \]
 with the following properties:
 \begin{itemize}
 \item[$(i)$] $\widetilde h$ is even, $\left.\wt
h\right|_{\G_+}=h\circ\phi_+,$
and $\left.\p_2\widetilde h\right|_{\G_+}=0;$
 \item[$(ii)$] If $h\in\clu_{s}$  and $\beta\in(0,1),$ then
$\Phi_h:=[(x,y)\longmapsto (x,y\widetilde h(x,y))]\in \mbox{\rm
Diff}^2(\Omega,\Omega(h))$.
 \end{itemize}
\end{lemma}

In a first step we use  the diffeomorphism  $\Phi_h$ to transform the scaled
problem \eqref{mbp2} into a nonlinear  and nonlocal evolution equation on
$\SSM,$ cf. \eqref{eq:CP}.
Therefore, we note that if $h_\eps:[0,T_\eps)\to \clu_s$ is a solution  of the
scaled  problem \eqref{mbp2}, then setting $w:=-v\circ \Phi_{h_\eps},$ we find
that the pair
 $(h_\eps,w)$ solves the  problem 
\begin{equation} \label{eq:TS2}
\left\{
\begin{array}{rlllll}
-\cla(\eps, h)w&=&0& \text{in}&\Omega,\\[1ex]
w&=&\displaystyle\kappa(\eps,h)\circ\phi_\pm&\text{on}& \Gamma_\pm,\\[1ex]
  \partial_{ t}  h&=&\clb(\eps,h)w& \text{on}&\SSM,
\end{array}
\right.
\end{equation}
where   $\cla:(0,1)\times \clu_s\to \cll(H^{s-3/2}(\Omega), H^{s-7/2}(\Omega))$
is the linear operator given by
\[
\cla(\eps, h)w:=D_iD_i w,
\]
with 
\[
\text{$D_1:=\eps(\p_1+a_1\p_2) $, \qquad  $D_2:=a_2\p_2$},
\]
and  $a_i,$ $i=1,2$  given by
\[
a_1:=-\frac{\p_1(y\wt h)}{\p_2(y\wt h)},  \qquad a_2:=\frac{1}{\p_2(y\wt
h)}.
\]

Furthermore, we define the boundary  operator 
$\clb: (0,1)\times\clu_s\to \cll(H^{s-3/2}(\Omega), H^{s-3}(\SSM))$
 by the relation 
\begin{align*}
\clb(\eps,h)w(x)
:=&\left(-h'\p_1(w\circ\Phi_h^{-1})+\eps^{-2}\p_2(w\circ\Phi_h^{-1}
)\right)(x , h(x))\\[1ex]
=& \left(- h'(\partial_1
w+a_1\partial_2w)+\eps^{-2}a_2\partial_2 w\right){(x,1)}, \qquad
x\in\SSM.
\end{align*}
It is not difficult to see that  $\clb(\eps,h)$ may be  also written as 
\[
\clb(\eps,h)w=\eps^{-2}\left.\left(\frac{a_{i,\eps}}{a_2}D_i
w\right)\right|_{\Gamma_+}\circ\phi_+^{-1}\qquad\mbox{on
$\SSM$,}
\]
where $a_{1, \eps}:=\eps a_1$ and $a_{2, \eps}:=a_2.$

Given $f\in H^{s-2}(\SSM)$ and $(\eps,h)\in(0,1)\times\clu_s$, we denote
throughout  this paper by $w(\eps,h)\{f\}$ the solution $w$
of the Dirichlet problem 
\begin{equation} \label{eq:DPDP}
\left\{
\begin{array}{rlllll}
-D_iD_iw&=&0& \text{in}&\Omega,\\[1ex]
w&=&f\circ\phi_\pm&\text{on}& \Gamma_\pm.
\end{array}
\right.
\end{equation}
With this notation, problem \eqref{eq:TS2} is equivalent to the abstract evolution equation
\begin{equation}\label{eq:CP}
\partial_t h=\clf(\eps, h)
\end{equation}
where we set 
\begin{equation}\label{defF}
F(\eps,h)\{f\}:=\clb(\eps, h)w(\eps,h)\{f\},
\end{equation}
 and the nonlinear
and
nonlocal operator 
$\clf:(0,1)\times\clu_s\to H^{s-3}(\SSM)$ is given by the relation 
\begin{equation}\label{defclf}
\clf(\eps,h):=F(\eps,h)\{\kappa(\eps,h)\},\qquad (\eps, h)\in(0,1)\times\clu_s.
\end{equation}
It will become clear from the considerations that follow that $w$, $F$, and
$\clf$
depend smoothly on their variables, i.e.
\begin{equation}\label{eq:regularity}
\begin{aligned}
&w\in C^\infty((0,1)\times\clu_s, \cll(H^{s-2}(\SSM),
H^{s-3/2}(\Omega))),\\[1ex]
&F\in C^\infty((0,1)\times\clu_s, \cll(H^{s-2}(\SSM),
H^{s-3}(\SSM))),\\[1ex]
&\clf\in C^\infty((0,1)\times\clu_s, H^{s-3}(\SSM)).
\end{aligned}
\end{equation}  

We start by estimating $w$ and its derivatives, and finish the  section by
proving estimates for the function $F.$
Some of the proofs rely on the following scaled version of the  integration by parts formula
\begin{equation}\label{IP}
\int_\0\frac{1}{a_2}vD_iw\, dx=-\int_\0\frac{1}{a_2}wD_iv\,dx
+\int_{\Gamma_+}\frac{a_{i,\eps}}{a_2}vw\,d\sigma
-\int_{\Gamma_-}\frac{a_{i,\eps}}{a_2}vw\, d\sigma,
\end{equation}
which is true for all  functions $v,w\in H^{1}(\0).$

In order to prove estimates for  the solution operator $w$ of \eqref{eq:DPDP}, 
we begin by analysing the solution operator corresponding to the same problem
when both equations in \eqref{eq:DPDP} have a nonzero right hand side.
As a first result we have:

\begin{prop}\label{P:1} There exist constants $\eps_0$, $C$ depending only on
$\clu_s$ and $s_0$ such that for integer $t\in[1,s_0+1/2]$, $f\in
H^{t-1/2}(\SSM)$, $f_i\in H^{t-1}(\0)$, $i=0,1,2$, and $(\eps,h)\in
(0,\eps_0)\times\clu_s$ the solution of the Dirichlet problem
\begin{equation} \label{eq:DP}
\left\{
\begin{array}{rlllll}
-D_iD_iw&=&f_0+\p_{i,\eps}f_i& \text{in} &\0,\\[1ex]
w&=&f\circ\phi_\pm& \text{on} & \Gamma_\pm,
\end{array}
\right.
\end{equation}
satisfies
\begin{equation}\label{G}
\|w\|^\0_{t,\eps}\leq
C\left(\sum_{i=0}^2\|f_i\|_{t-1}^\0+\|f\|_{t-1}+\eps^{1/2}\|
f\|_{t-1/2}\right).
\end{equation}
Additionally,
\begin{equation}\label{F}
\|\p_2 w\|^{\p\0}_{-1/2}\leq
C\left(\sum_{i=0}^2\|f_i\|_{0}^\0+\|f_2\|^{\p\0}_{-1/2}+\|
f\|_{0}+\eps^{1/2}\| f\|_{1/2}\right).
\end{equation}
\end{prop}

\begin{proof} 
{\it Step $1$.} We show \eqref{G} for $t=1$. We will consider the case $f=0$
first.

Using relation \eqref{IP}, we proceed as in \cite[Lemma 3.2]{guprth} and find 
\begin{align*}
I:=&\int_\0\frac{1}{a_2}D_iwD_iw\, dx=-\int_\0\frac{1}{a_2}wD_iD_iw\,
dx=\int_\0\frac{1}{a_2}wf_0\, dx+\int_\0\frac{1}{a_2}w\p_{i,\eps}f_i\, dx\\[1ex]
=&\int_\0\frac{1}{a_2}wf_0\, dx-\int_\0\frac{1}{a_2}\p_{i,\eps}wf_i\,
dx-\int_\0\p_{i,\eps}\left(\frac{1}{a_2}\right)wf_i\, dx,
\end{align*}
where we used integration by parts to obtain the last equality. So
\begin{align*}
I&\leq C\sum_{i=0}^2\|f_i\|_0^\0\|w\|^\0_{1,\eps}.
\end{align*}
On the other hand,
\begin{equation*}
I\geq c \int_\0 |\nabla_\eps w|^2=c\left(\|\nabla_\eps w\|_0^\0\right)^2,
\end{equation*}
provided $\eps\in(0,\eps_0)$ and $\eps_0$ is sufficiently small (with a constant
$c$ independent of $\eps$).
Using Poincar\'e's inequality \eqref{eq:Poin1}, the estimate follows.

If $f\neq 0$, we let $z:=w-\wt f\in H^{t}(\0),$ where $\wt f:=Ef$.
Then $z=0$ on $\p\0$ and $z$ solves in $\0$ the equation $-D_iD_iz=\bar
f_0+\p_{i,\eps}\bar f_i$,
where
\begin{equation}\label{eq:fct}
\bar f_0:=f_0-\eps\p_2a_1D_1\wt f-\p_2a_2D_2\wt f, \ \ \bar
f_{1}=f_{1}+D_1\wt
f, \ \ \bar f_{2}=f_2+\eps a_1D_1\wt f+a_2D_2\wt f.
\end{equation}
Using \eqref{eq:Esp2} and the result for homogeneous boundary data, we
conclude that \eqref{G} holds with $t=1$.

{\it Step $2$.} We show \eqref{F}. Define 
\be\label{deftb}
\widetilde\clb(\eps,h)w:=\pm\eps^{-2}\frac{a_{i,\eps}}{a_2}D_iw\quad\text{
on }\Gamma_\pm
\ee
and observe
\begin{equation}\label{repd2}
\p_2w=\pm
\eps^2h(1+\eps^2{h'}^2)^{-1}(\widetilde\clb(\eps,h)w+h'\p_1w)\quad\mbox{on
$\Gamma_\pm$.}
\end{equation}
We start with the case $f=0$ again. 
Then $\partial_1w=0$ and thus
it is sufficient to estimate
$\eps^2\|\widetilde \clb(\eps,h)w\|_{-1/2}^{\p\0}$. For this purpose, pick
$\psi\in
H^{3/2}(\partial\Omega)$ and define $u\in H^2(\Omega)$ to be the solution of
the Dirichlet problem
\begin{equation*}
\left\{
\begin{array}{rlllll}
-D_iD_iu&=&0& \text{in}&\0,\\[1ex]
u&=&\psi&\text{on}& \p\0.
\end{array}
\right.
\end{equation*}
Then, by the transformed version of Green's second identity, 
\begin{align*}
&\eps^2\int_{\p\0}\widetilde \clb(\eps,h)w\psi\,d\sigma=\int_{\0}\frac{1}{a_2}
D_iD_iwu\ ,
dx= -\int_{\0}u\frac{f_0+\eps\partial_1f_1+\partial_2f_2}{a_2}\,dx\\
=&\int_{\0}\left[\eps \p_1\left(\frac{u}{a_2}\right) 
f_1+\p_2\left(\frac{u}{a_2}\right)f_2-\frac{uf_0}{a_2}\right]\,
dx+\int_{\Gamma_-}
\frac{u}{a_2}f_2d\sigma-\int_{\Gamma_+}\frac{u}{a_2}f_2\,d\sigma.
\end{align*}
Consequently, applying the result of Step 1 to $u$, 
\begin{align*}
\eps^2\left|\int_{\p\0}\widetilde\clb(\eps,h)w\psi\,d\sigma\right|&\leq C
\left(\|u\|_{1,\eps}^\0\sum_{i=0}^2\|f_i\|_0^\0+\|f_2\|^{\p\0}_{-1/2}
\|\psi\|_{1/2}^{\p\0}\right)\\
&\leq
C\left(\|f_2\|^{\p\0}_{-1/2}+\sum_{i=0}^2\|f_i\|_0^\0\right) 
\|\psi\|_{1/2}^{\p\0}.
\end{align*}
This implies \eqref{F} for $f=0$. To treat the general case, define $\tilde f$,
$\bar f_i$ and $z$ as in Step 1. Then, by the preliminary result,
\begin{align*}
\|\p_2w\|_{-1/2}^{\p\0}&\leq \|\p_2z\|_{-1/2}^{\p\0}+\|\p_2Ef\|_{-1/2}^{\p\0},\\
\|\p_2z\|_{-1/2}^{\p\0}&\leq \textstyle{\sum_{i=0}^2}\|\bar f_i\|_0^\0
+\|\bar f_2\|_{-1/2}^{\p\0},\\
\textstyle{\sum_{i=0}^2}\|\bar f_i\|_0^\0&\leq
\textstyle{\sum_{i=0}^2}\|f_i\|_0^\0+C\|Ef\|_{1,\eps}^\0,\\
\|\bar f_2\|_{-1/2}^{\p\0}&\leq \|
f_2\|_{-1/2}^{\p\0}+C\left(\eps^2\|f\|_{1/2}^{\p\0}+\|\p_2Ef\|_{-1/2}^{\p\0}
\right),
\end{align*}
and the result follows from \eqref{eq:Esp2} and \eqref{eq:tr-1/2-2}.

{\it Step $3$.} We prove \eqref{G}  by induction over $t$. The case $t=1$
has been treated in Step 1. Assume now \eqref{G} for an integer
$t\in[1,s_0-1/2]$.
Differentiating both equations of  \eqref{eq:DP} with respect to $x$ we find
 that $\partial_1 w$ satisfies
\begin{equation*}
\left\{\begin{array}{rcll}
-D_iD_i \partial_1 w&=&\bar f_0+\partial_{i,\eps}\bar f_i&\mbox{in $\0$},\\
\partial_1 w&=&  f'\circ\phi_\pm
&\mbox{on $\Gamma_\pm$,}
\end{array}
\right.
\end{equation*}
where
\begin{align*}
\bar f_0&=\p_1f_0-\partial_{12}a_{i,\eps}D_i
w-\partial_1a_{i,\eps}\partial_2a_{i,\eps}\partial_2 w,\\
\bar  f_1&= \p_1f_1+\eps\partial_1a_1\partial_2w,\\
\bar
f_2&=\p_1f_2+\partial_1a_{i,\eps}D_iw+a_{i,\eps}\partial_1a_{i,\eps}\partial_2w.
\end{align*}
Using this and the induction assumption, we conclude that 
\begin{equation}\label{eq:G1}
\|\partial_ 1w\|_{t,\eps}^\0\leq C\left(\sum_{i=0}^2\|f_i\|_{t}^\0+\|
f\|_{t}+\eps^{1/2}\| f\|_{t+1/2}\right).
\end{equation}
In order to estimate   $\|\p_{22}w\|_{t-1}^\0$,
we use the first equation of \eqref{eq:DP} and the explicit representation
\begin{equation}\label{repA}
\cla(\eps, h)w:=\eps^2 \p_{11}w+2\eps^2a_1\p_{12}w+(\eps^2
a_1^2+a_2^2)\p_{22}w+(\eps^2\p_1a_1+\eps^2a_1\p_2a_1+a_2\p_2a_2)\p_2w
\end{equation}
to obtain
\begin{align*}
\p_{22}w&=\frac{
f_0+\p_{i,\eps}f_i-\eps^2\p_{11}w-2\eps^2a_1\p_{12}w-(\eps^2\p_1a_1
+\eps^2a_1\p_2a_1+a_2\p_2a_2)\p_2w}{\eps^2a_1^2+a_2^2},
\end{align*}
and  see that
\begin{equation}\label{eq:G2}
\|\p_{22}w\|_{t-1}^\0\leq C\left(\sum_{i=0}^2
\|f_i\|_{t}^\0+\|\p_1w\|_{t,\eps}^\0+\|w\|_{t,\eps}^\0\right).
\end{equation}
Combining \eqref{eq:G1}, \eqref{eq:G2},  the induction assumptions, and the relation 
\[
\|w\|_{t+1,\eps}^\0\leq
C\left(\|w\|_{t,\eps}^\0+\|\p_1w\|_{t,\eps}^\0+\|\p_{22}w\|_{t-1}^\0\right)
\]
yields the desired estimate for $\|w\|_{t+1,\eps}^\0.$ 
This completes the proof.
\end{proof}

Using this result we can additionally show that 
then
\begin{equation}\label{e:G+H-dom}
\|\partial_2w(\eps,h)\{f\}\|^{\0}_{s_0-1/2}\leq
C\eps^2\|f\|_{s_0+3/2}.
\end{equation}
(Note that this involves a higher norm of $f$, but the constant involved in
the estimate is of order $\eps^2$.)

To show this, let $\phi\in H^{s-2}(\0)$ be the extension of $f$ given by 
$\phi(x,y)=f(x)$ and define $z:=w(\eps,h)\{f\}-\phi$. Then

\begin{equation*}
\left\{
\begin{array}{rlllll}
-D_iD_iz&=&\eps^2\partial_{11}\phi& \text{in} &\0,\\[1ex]
z&=&0& \text{on} & \p\0,
\end{array}
\right.
\end{equation*}
and by the unscaled trace inequality and $\eqref{G}$ with $t=s_0+1/2$ we get
\begin{equation*}
\|\p_2w\|^{\0}_{s_0-1/2}= \|\p_2z\|^{\0}_{s_0-1/2}
\leq C\|z\|_{s_0+1/2,\eps}^{\0}\leq C\eps^2\|\p_{11}\phi\|_{s_0-1/2}^\0
\end{equation*}
and therefore \eqref{e:G+H-dom}.
In particular, this implies by the unscaled trace estimate
\begin{equation}\label{e:G+H}
\|\partial_2w(\eps,h)\{f\}\|^{\p\0}_{s_0-1}\leq
C\eps^2\|f\|_{s_0+3/2}.
\end{equation}

Next, we prove a coercivity estimate for the scaled
Dirichlet-Neumann   operator $F(\eps,h),$ $(\eps,h)\in (0,\eps_0)\times\clu_s,$
which will be a key point in the proof of Theorem \ref{thm2}. 
Given $\varphi\in H^{1/2}(\SSM)$ and $\eps>0,$ we set
\[
\|\varphi\|_{1/2, \eps}:=\|\varphi\|_0+\eps^{1/2}\|\varphi\|_{1/2}.
\]

\begin{lemma}\label{L:ana}
There exists a positive constant $c$ such that for all $(\eps,h)\in
(0,\eps_0)\times\clu_s$  and $\varphi\in H^{3/2}(\SSM)$
which satisfy
\begin{align*}
\int_\SSM\varphi\, dx=0
\end{align*}
we have 
\begin{equation}\label{eq:ana}
\langle F(\eps,h)\{\varphi\}|\varphi\rangle_{L^2(\SSM)}\geq
c\|\varphi\|^2_{1/2,\eps}.
\end{equation}
\end{lemma}
\begin{proof} 
Let $w:=w(\eps,h)\{\varphi\}\in H^2(\0)$, recall the definition of 
$\tilde\clb(\eps,h)w$ from \eqref{deftb} and observe that due to symmetry
\[\tilde\clb(\eps,h)w(x,1)=\tilde\clb(\eps,h)w(x,-1)\qquad x\in\SSM.\]

Using \eqref{IP}, we have
\begin{align*}
&\langle F(\eps,h)\{\varphi\}|\varphi\rangle_{L^2(\SSM)} =\int_{\G_+} w
\tilde\clb(\eps,h)w\, d\sigma=\eps^{-2}\int_{\G_+} \frac{a_{i,\eps}}{a_2}wD_iw\,
d\sigma\\[1ex]
=&\frac{\eps^{-2}}{2}\int_{\G_+}\frac{a_{i,\eps}}{a_2}wD_iw \,
d\sigma-\frac{\eps^{-2}}{2}\int_{\G_-}\frac{a_{i,\eps}}{a_2}wD_iw \,
d\sigma
=\frac{\eps^{-2}}{2}\int_{\0}\frac{1}{a_2}D_i wD_iw\, dx\\
 \geq&
c\eps^{-2}\int_{\0}|\nabla_\eps w|^2\, dx=c\eps^{-2}\left(\|\nabla_\eps
w\|_0^\0\right)^2,
\end{align*} 
cf. Proposition \ref{P:1}.
From Poincar\'e's  inequality \eqref{eq:Poin2}  together
with \eqref{eq:scest} we obtain the desired estimate.
\end{proof}

Now we prove estimates for the Fr\'echet derivatives of the solution
$w=w(\eps,h)\{f\}$ of \eqref{eq:DPDP} with respect to $h$.
The results established in Proposition \ref{P:1} will be used as basis for
an induction argument.

\begin{prop}\label{P:2} 
Given $k\in\NNM$,  $h_1,\ldots, h_k\in H^{s}(\SSM),$ and $f\in H^{s-2}(\SSM)$,
the Fr\'echet derivative $w^{(k)}:=w^{(k)}(\eps,h)[h_1,\ldots,h_k]\{f\}$
satisfies 
\begin{equation}\label{C}
\|w^{(k)}\|_{t,\eps}^\0\leq C\|h_1\|_{s_0}\ldots 
\|h_k\|_{s_0}\|f\|_{t-1/2}.
\end{equation} 
for all integer $t\in[1,s_0-1/2]$. Additionally,
\begin{equation}\label{estd2der}
\|\p_2 w^{(k)}\|^{\p\0}_{-1/2} \leq
C\|h_1\|_{s_0}\ldots \|h_k\|_{s_0}\|f\|_{1/2}.
\end{equation}
The constant $C$ depends only on $k$, $s_0$, and $\clu_s$. 
\end{prop}
\begin{proof} 
We prove both estimates by induction over $k$. For $k=0$ they hold due to
Proposition \ref{P:1}. Assume now \eqref{C}, \eqref{estd2der} for all Fr\'echet
derivatives up to order $k$. 
Differentiating \eqref{eq:DPDP}   $(k+1)-$times with respect to $h$, yields that
$w^{(k+1)}$ is the solution of
\begin{equation}\label{eq:DGF}
\left\{
\begin{array}{rlllll}
-D_iD_iw^{(k+1)}&=&\sum\limits_{\sigma\in
S_{k+1}}\sum\limits_{l=0}^{k}C_l\p_h^{l+1}  \cla(\eps,h)[h_{\sigma(1)},\ldots
h_{\sigma(l+1)}]w^{(k-l)}& \text{in} \ \0,\\[1ex]
w^{(k+1)}&=&0&\text{on} \ \p\0,
\end{array}
\right.
\end{equation}
where
$w^{(k-l)}=w^{(k-l)}(\eps,h)[h_{\sigma(l+2)},\ldots,h_{\sigma(k+1)}]\{f\}$
 and $S_{k+1}$ is the set of permutations of $\{1,\ldots,k+1\}$.

We are going to define functions $F_i$ in $\Omega$ such that the right hand side
in \eqref{eq:DGF}$_1$ can be written as
\[\sum\limits_{\sigma\in
S_{k+1}}\sum\limits_{l=0}^{k}C_l\p_h^{l+1}  \cla(\eps,h)[h_{\sigma(1)},\ldots
h_{\sigma(l+1)}]w^{(k-l)}=F_0+\p_{i,\eps}F_i.\]
The functions $F_i$ are sums of terms to be specified below. For this purpose,
we recall \eqref{repA} and consider the Fr\'echet derivatives of the ocurring
terms separately.
\begin{itemize}
\item[(i)] When differentiating $\eps^2 \p_{22}w$ we do not obtain  any term  on
the right hand side of the first equation of \eqref{eq:DGF}.
\item[(ii)] The terms on the right hand side of the first equation  of 
\eqref{eq:DGF} which are obtained by differentiating  $2\eps^2a_1\p_{12}w$   may
be written as follows:
\[
\eps^2 a_1^{(l+1)}\p_{12}w^{(k-l)}
=\p_{1,\eps}\left[\eps  a_1^{(l+1)}\p_2w^{(k-l)}\right]
-\p_{1}\left(\eps^2  a_1^{(l+1)}\right)\p_2w^{(k-l)},
\]
where $a_1^{(l+1)}:=a_1^{(l+1)}(h)[h_{\sigma(1)},\ldots
h_{\sigma(l+1)}]$.
The last term belongs to $F_0,$ while the one in the square brackets belongs
 to  $F_{1}$.
\item[(iii)] When differentiating $(\eps^2 a_1^2+a_2^2)\p_{22}w$ we obtain terms
of the  form
\[
(\eps^2
a_1^2+a_2^2)^{(l+1)}\p_{22}w^{(k-l)}
=\p_{2,\eps}\left[  (\eps^2
a_1^2+a_2^2)^{(l+1)}\p_2w^{(k-l)}\right]
-\p_{2}  (\eps^2
a_1^2+a_2^2)^{(l+1)}\p_2w^{(k-l)},
\]
where $(\eps^2 a_1^2+a_2^2)^{(l+1)}:=\p_h^{l+1} (\eps^2
a_1^2+a_2^2)(h)[h_{\sigma(1)},\ldots h_{\sigma(l+1)}]$.
The last term belongs to $F_0$ while the expression  in the square brackets
belongs to  $F_{2}$.
\item[(iv)] All  terms corresponding to
$(\eps^2\p_1a_1+\eps^2a_1\p_2a_2+a_2\p_2a_2)\p_2w$ are absorbed by
$F_0$.
\end{itemize}
Summarizing, we get
\begin{equation}\label{eq:DGF2}
\left\{
\begin{array}{rlllll}
-D_iD_iw^{(k+1)}&=&F_0+\p_{i,\eps}F_i& \text{in} \ \0,\\[1ex]
w^{(k+1)}&=&0&\text{on} \ \p\0,
\end{array}
\right.
\end{equation}
where 
\[F_i=\sum_{\sigma\in
S_{k+1}}\sum_{l=0}^k\alpha_{li}[h_{\sigma(1)},\ldots,h_{\sigma(l+1)}]
\p_2w^{(k-l)}\]
and 
\[\alpha_{li}[H_1,\ldots,H_{l+1}]=\sum_{\sigma'\in
S_{l+1}}\beta_{li}(\eps,\widetilde h)\p^{\gamma_{li,1}}\widetilde{
H}_{\sigma'(1)}\ldots
\p^{\gamma_{li,l+1}}\widetilde{H}_{\sigma'(l+1)}\]
with smooth functions $\beta_{li}$ and $|\gamma_{l0,j}|\in\{0,1,2\}$, 
$|\gamma_{l1,j}|,|\gamma_{l2,j}|\in\{0,1\}$. 
Fixing $\sigma$ and $l$, writing
$\alpha_{li}:=\alpha_{li}[h_{\sigma(1)},\ldots,h_{\sigma(l+1)}]$
and using the induction assumption we estimate
\begin{align*}
\|\alpha_{li}\p_2w^{(k-l)}\|_{t-1}^\0
\leq C\|\alpha_{li}\|_{s_0-3/2}^\0\|\p_2w^{(k-l)}\|_{t-1}^\0
&\leq C\|\widetilde h_{\sigma(1)}\|^\0_{s_0+1/2}\ldots
\|\widetilde h_{\sigma(l+1)}\|^\0_{s_0+1/2}\|w^{(k-l)}\|^\0_{t,\eps}\\
&\leq C\|h_1\|_{s_0}\ldots\|h_{k+1}\|_{s_0}\|f\|_{t-1/2}.
\end{align*}
Thus
\begin{equation}\label{eq:Fies}
\|F_i\|_{t-1}^\0\leq
C\|h_1\|_{s_0}\ldots\|h_{k+1}\|_{s_0}\|f\|_{t-1/2},\qquad
i=0,1,2,
\end{equation}
and \eqref{C} (with $k$ replaced by $k+1$) follows from \eqref{G}.

Similarly, 
\begin{align*}
\|\alpha_{l2}\p_2w^{(k-l)}\|_{-1/2}^{\p\0}
&\leq C\|\alpha_{l2}\|_{s_0-1}^{\p\0}\|\p_2w^{(k-l)}\|_{-1/2}^{\p\0}
\leq C\|\alpha_{l2}\|_{s_0-1/2}^{\0}\|\p_2w^{(k-l)}\|_{-1/2}^{\p\0}\\
&\leq C \|\widetilde h_{\sigma(1)}\|^\0_{s_0+1/2}\ldots
\|\widetilde h_{\sigma(l+1)}\|^\0_{s_0+1/2}\|\p_2w^{(k-l)}\|_{-1/2}^{\p\0}\\
&\leq C\|h_1\|_{s_0}\ldots\|h_{k+1}\|_{s_0}\|f\|_{t-1/2}.
\end{align*}
Therefore
\[\|F_2\|_{-1/2}^{\p\0}\leq C
\|h_1\|_{s_0}\ldots\|h_{k+1}\|_{s_0}\|f\|_{1/2},\]
and \eqref{estd2der} (with $k$ replaced by $k+1$) follows from \eqref{eq:Fies}
with $t=1$ and \eqref{F}.
\end{proof}

 We prove now an estimate similar to \eqref{estd2der} which is  optimal
 with respect to one of the ``variations'' $h_k$ (say $h_1$). The price to
pay here is a stronger norm for $f$.

\begin{prop}\label{P:3}
Under the assumptions of Proposition {\rm\ref{P:2}} we additionally have 
\[\|\p_2w^{(k)}\|_{-1/2}\leq
C\|h_1\|_{3/2}\|h_2\|_{s_0}\ldots\|h_k\|_{s_0}\|f\|_{s_0-1}.\]
The constant $C$ depends only on $k$, $s_0$, and $\clu_s$.
\end{prop}
\begin{proof}
We show the more general estimate
\begin{equation}\label{estind2}
\|w^{(k)}\|^\0_{1,\eps}+\|\p_2w^{(k)}\|_{-1/2}\leq
C\|h_1\|_{3/2}\|h_2\|_{s_0}\ldots\|h_k\|_{s_0}\|f\|_{s_0-1}
\end{equation}
by induction over $k$. For $k=0$ the statement is contained in Proposition
\ref{P:1}. Assume now \eqref{estind2} for all derivatives up to some order
$k$. We proceed as in the proof of Proposition \ref{P:2}, reconsider
problem \eqref{eq:DGF2} and have to show now 
\[\|F_i\|_0^\0,\;\|F_2\|^{\p\0}_{-1/2}\leq C
\|h_1\|_{3/2}\|h_2\|_{s_0}\ldots\|h_k\|_{s_0}\|f\|_{s_0-1}.\]
For this purpose, we fix $\sigma$ and $l$ and estimate 
\[\|\alpha_{li}\p_2w^{(k-l)}\|_0^\0,\;\|\alpha_{l2}\p_2w^{(k-l)}\|^{\p\0}_{-1/2}
.
\]
We have to distinguish two cases, depending on whether the argument $h_1$
occurs in the first or in the second factor.

Case 1: $\sigma^{-1}(1)\leq l+1$. 
Using Proposition \ref{P:2} with $t=s_0-1/2$ we estimate
\begin{align*}
\|\alpha_{li}\p_2w^{(k-l)}\|_0^\0&\leq
C\|\alpha_{li}\|_0^\0\|\p_2w^{(k-l)}\|_{s_0-3/2}^\0
\leq C\|\widetilde h_1\|_2^\0
\prod
\|\widetilde h_{\sigma(j)}\|_{s_0+1/2}^\0\,\|w^{(k-l)}\|^\0_{s_0-1/2,\eps}\\
&\leq C\|h_1\|_{3/2}\|h_2\|_{s_0}\ldots\|h_k\|_{s_0}\|f\|_{s_0-1}.
\end{align*}
The product is taken over $j\in\{1,\ldots,l+1\}\setminus\{\sigma^{-1}(1)\}$.
Similarly,
\begin{align*}
&\|\alpha_{l2}\p_2w^{(k-l)}\|_{-1/2}^{\p\0}\leq C
\|\alpha_{l2}\|_{1/2}^{\p\0}\|\p_2w^{(k-l)}\|_{s_0-2}^{\p\0}\leq C
\|\alpha_{l2}\|_{1}^{\0}\|\p_2w^{(k-l)}\|_{s_0-3/2}^{\0}\\
\leq& C\|\widetilde h_1\|_2^\0
\prod
\|\widetilde h_{\sigma(j)}\|_{s_0+1/2}^\0\,\|w^{(k-l)}\|^\0_{s_0-1/2,\eps}
\leq C\|h_1\|_{3/2}\|h_2\|_{s_0}\ldots\|h_k\|_{s_0}\|f\|_{s_0-1}.
\end{align*}

Case 2: $\sigma^{-1}(1)\leq l+1$. We apply the induction assumption and
estimate 
\begin{align*}
\|\alpha_{li}\p_2w^{(k-l)}\|_0^\0&\leq
C\|\alpha_{li}\|^\0_{s_0-3/2}\|\p_2w^{(k-l)}\|^\0_0
\leq C\|\widetilde h_{\sigma(1)}\|^\0_{s_0+1/2}\ldots
\|\widetilde h_{\sigma(l+1)}\|^\0_{s_0+1/2} \|w^{(k-l)}\|^\0_{1,\eps}\\
&\leq C\|h_1\|_{3/2}\|h_2\|_{s_0}\ldots\|h_k\|_{s_0}\|f\|_{s_0-1}.
\end{align*}
Similarly,
\begin{align*}
&\|\alpha_{l2}\p_2w^{(k-l)}\|_{-1/2}^{\p\0}\leq C
\|\alpha_{l2}\|_{s_0-2}^{\p\0}\|\p_2w^{(k-l)}\|_{-1/2}^{\p\0}
\leq C\|\alpha_{l2}\|_{s_0-3/2}^{\0}\|\p_2w^{(k-l)}\|_{-1/2}^{\p\0}\\
\leq& C\|\widetilde h_{\sigma(1)}\|^\0_{s_0+1/2}\ldots
\|\widetilde h_{\sigma(l+1)}\|^\0_{s_0+1/2} \|w^{(k-l)}\|^\0_{1,\eps}
\leq C\|h_1\|_{3/2}\|h_2\|_{s_0}\ldots\|h_k\|_{s_0}\|f\|_{s_0-1}.
\end{align*}
The proof is completed now by carrying out the summations over $\sigma$ and
$l$ and applying Proposition \ref{P:1} to \eqref{eq:DGF2}.
\end{proof}

We recall
(cf. \eqref{defF})
\[
F(\eps,h)\{f\}=\frac{1}{h}(\eps^{-2}+{h'}^2)(\p_2w(\eps,h)\{f\})|_{\Gamma_+}
\circ\phi_+^{-1} -h'f'.
\]
Applying the product rule of differentiation and product estimates as above we
find from this and Propositions \ref{P:2} and \ref{P:3} 
\begin{align}\label{eq:cru1}
 \|F^{(m)}(\eps,h)[h_1,\ldots h_m]\{f\}\|_{-1/2} &\leq
C\eps^{-2}\|h_1\|_{s_0}\ldots\|h_k\|_{s_0}\|f\|_{1/2},\\[1ex]
 \label{eq:cru2}
  \|F^{(m)}(\eps,h)[h_1,\ldots h_m]\{f\}\|_{-1/2} &\leq
C\eps^{-2}\|h_1\|_{3/2}\ldots\|h_k\|_{s_0}\|f\|_{s_0-1}
\end{align}
for all $m\in\NNM,$ $(\eps,h)\in(0,\eps_0)\times\clu_s,$ $f\in H^{s-2}(\SSM)$ , 
and $h_1,\ldots h_m\in H^{s}(\SSM).$ Additionally, using \eqref{e:G+H},
\[
\|F(\eps,h)\{f\}\|_{s_0-1} \leq C\|f\|_{s_0+3/2}. 
\]
In particular, we have
\begin{equation}\label{eq:Boun}
\|\clf(\eps,h)\|_{s_0-1}\leq C\|h\|_{s_0+7/2}\leq C,
\qquad(\eps,h)\in(0,\eps_0)\times \clu_s.
\end{equation}
The constants depend only upon ${\mathcal U}_s$, $s_0$, and $m$.

Moreover, we obtain:
\begin{lemma}\label{L:32} 
Given  $h_1\in H^{s}(\SSM)$ and $f\in H^{s-2}(\SSM)$,  we have
\begin{equation}\label{CCC}
\|F'(\eps,h)[h_1]\{f\}\|_{-1/2}\leq C\|h_1\|_{3/2}\|f\|_{s_0+3/2}.
\end{equation}
\end{lemma}
\begin{proof} For brevity we write $w':=w'(\eps,h)\{f\}$. 
Differentiating \eqref{eq:DPDP} with respect to $h$ yields 
\begin{equation*}
\left\{
\begin{array}{rllllll}
-D_iD_i w'&=&f_0+\p_{i,\eps}f_i\quad \text{in} \ \0,\\[1ex]
 w'&=& 0\quad \text{on} \ \p\0,
\end{array}
\right.
\end{equation*}
where 
\begin{align*}
f_0:=&\eps^2\p_ha_1[h_1]\p_{12}w+2^{-1}\p_h(\eps^2 a_1^2+a_2^2)[h_1]\p_{22}w,\\
f_1:=&\eps\p_ha_1[h_1]\p_2w,\\
f_2:=&2^{-1}\p_h(\eps^2 a_1^2+a_2^2)[h_1]\p_2w,\\
w:=&w(\eps,h)\{f\}.
\end{align*}
By \eqref{G} and \eqref{e:G+H-dom} we have
\begin{align*}
\|\p_2 w'\|_{1/2}\leq&\| w'\|_{2,\eps}^\0\leq C\sum_{i=0}^2\|f_i\|_1^\0\leq
C\|h_1\|_{3/2}\|\p_2 w\|_{3}^{\Omega}
\leq C\eps^2\|h_1\|_{3/2}\|f\|_{s_0+3/2}.
\end{align*}
The result follows easily from this. 
\end{proof}

Next we give an estimate for the remainder term that occurs when curvature
differences are linearized.

\begin{lemma}\label{L:5} Let  $\eps\in(0,1)$ and $h,\overline h\in\clu_s\cap
H^{s+3/2}(\SSM)$.
Then
\begin{equation}\label{eq:curvature} 
\left\|\p_x^{s-1}\left(\kappa(\eps,h)-\kappa(\eps,\ov h)\right)-
\kappa'(\eps,h)[(h-\overline h)^{(s-1)}]\right\|_{1/2}\leq
C(1+\|\overline h\|_{s+3/2})\|h-\overline h\|_{s+1/2}.
 \end{equation}
The constant $C$ depends only on ${\mathcal U}_s$.
\end{lemma}
\begin{proof} By the chain rule, 
\begin{align*}
\p_x^{s-1}\kappa(\eps,h)&=\kappa'(\eps, h)[h^{(s-1)}]
+\sum_{l=2}^{s-1}C_{p_1\ldots p_l}\kappa^{(l)}(\eps,h)[h^{(p_1)}, \ldots,
h^{(p_l)}],
\end{align*}
with  $1\leq p_1\leq \ldots\leq p_l,$ and
$p_1+\ldots+p_l=s-1.$ 
This also holds  if we replace $h$ by $\overline h$.  
We subtract these identities and obtain 
\begin{align}\label{rkappa}
&\p_x^{s-1}\left(\kappa(\eps,h)-\kappa(\eps,\ov h)\right)-
\kappa'(\eps,h)[(h-\overline
h)^{(s-1)}]\nonumber\\
=&\left(\kappa'(\eps,h)-\kappa'(\eps,\ov h)\right)[\ov
h^{(s-1)}]
+\sum_{l=2}^{s-1} C_{p_1\ldots p_l}\bigg(
(\kappa^{(l)}(\eps,h)-\kappa^{(l)}(\eps,\ov h))[\ov h^{(p_1)}, \ldots,\ov
h^{(p_l)}]\nonumber\\
&+\sum_{j=1}^l\kappa^{(l)}(\eps,h)[h^{(p_1)},\ldots,h^{(p_{j-1})},
(h-\ov h)^{(p_j)},\ov h^{(p_{j+1})},\ldots,\ov h^{(p_l)}]\bigg)
\end{align}
The terms on the right are estimated separately. One straightforwardly gets
 \begin{equation}\label{eq:cur1}
 \|\kappa^{(l)}(\eps,h)[h_1,\ldots ,h_l]\|_{1/2}\leq
C\|h_1\|_{3}\ldots\|h_{l-1}\|_{3}\|h_l\|_{5/2}
\end{equation}
for all $l\in\NNM$, $l\geq1$, $h_l\in H^{5/2}(\SSM)  $ and $ h_1, \ldots
h_{l-1}\in
H^{3}(\SSM)$.
Since $\clu_s$ is convex, we additionally have
\begin{align*}
&\left\|\left(\kappa^{(l)}(\eps,h)-\kappa^{(l)}(\eps,\ov
h)\right)[h_1,\ldots,h_l]\right\|_{1/2}\\
\leq&\int_0^1\left\|\kappa^{(l+1)}(\eps,rh+(1-r)\ov h)
[h-\ov h,h_1,\ldots,h_l]\right\|_{1/2}\,dr\leq
C\|h-\ov h\|_3\|h_1\|_3\ldots\|h_{l-1}\|_3\|h_l\|_{5/2}.
\end{align*}
Applying these estimates to all terms in \eqref{rkappa} and adding them up
yields the result.
\end{proof}

Finally, we give a parallel estimate concerning the complete operator $\mathcal
F$. 
Using the invariance of our problem with respect to horizontal translations we
obtain, as in \cite{guprth}, Eq. (6.8), the ``chain rule''

\begin{equation}\label{eq:declf}
\p_x^{s-1}\clf(\eps,h)=F(\eps,h)\{\p_x^{s-1}\kappa(\eps,h)\}+\sum_{k\geq 1}
C_{p_1,\ldots,
p_{k+1}}F^{(k)}(\eps,h)[h^{(p_1)},\ldots,h^{(p_k)}]\{\p_x^{p_{k+1}}\kappa(\eps,
h)\},
\end{equation}
 $h\in\clu_s$ sufficiently smooth.
The sum is taken over all $(k+1)$-tuples $(p_1, \ldots,p_{k+1})$ satisfying 
$p_1+\ldots+p_{k+1}=s-1$ and  $p_1, \ldots, p_k\geq 1$.

\begin{lemma}\label{L:6}  Additionally to Lemma {\rm\ref{L:5}}, assume
$\eps\in(0,\eps_0)$.
Define
\[
P_s(\eps,h,\ov h):=\p_x^{s-1}\left(\clf(\eps,h)-\clf(\eps,\ov h)\right)
-F(\eps,h)\{\kappa'(\eps, h)[(h- \ov h)^{(s-1)}]\}.
\]
Then
\begin{equation}\label{eq:L6}
\|P_s(\eps, h, \ov h)\|_{-1/2}\leq C \eps^{-2}(1+\|\ov h\|_{s+3/2})\|h-\ov
h\|_{s+1/2}.
 \end{equation}
The constant $C$ depends only on $\clu_s$.
\end{lemma}
\begin{proof}
Observe that by density and continuity arguments, it is sufficient to show
\eqref{eq:L6} under the additional assumption that $h$ and $\ov h$ are smooth.
We infer from \eqref{eq:declf} that $P_s(\eps,h,\ov h)=E^a+E^b+G,$ with 
\begin{align*}
E^a:=&F(\eps,h)\{\p_x^{s-1}\kappa(\eps,h)\} -F(\eps,\ov
h)\{\p_x^{s-1}\kappa(\eps,\ov h)\} 
-F(\eps,h)\{\kappa'(\eps, h)[(h- \ov h)^{(s-1)}]\}\\[1ex]
E^b:=&F'(\eps,h)[h^{(s-1)}]\{\kappa(\eps,h)\}-F'(\eps,\ov h) [\ov
h^{(s-1)}]\{\kappa(\eps,\ov h)\}
\end{align*}
and $G:=\sum C_{p_1,\ldots, p_{k+1}} E^c_{p_1,\ldots, p_{k+1}},$ where  the sum
is taken over all tuples satisfying additionally $1\leq p_{k+1}\leq s-2,$ 
and  $E^c_{p_1,\ldots, p_{k+1}}=:E^c$ is given by
\begin{align*}
 E^c:=&F^{(k)}(\eps,h)[h^{(p_1)},\ldots,h^{(p_k)}]
\{\p_x^{p_{k+1}}\kappa(\eps,h)\}-F^{(k)}(\eps,\ov h) 
[\ov h^{(p_1)},\ldots,\ov h^{(p_k)}]
\{\p_x^{p_{k+1}}\kappa(\eps,\ov h)\}.
\end{align*}
We estimate $E^a$ first and write $E^a=E^a_1+E^a_2,$ where
\begin{align*}
E^a _1:=&F(\eps,h)\{\p_x^{s-1}(\kappa(\eps,h)-\kappa(\eps,\ov h))
-\kappa'(\eps, h) [(h- \ov h)^{(s-1)}]\}\\[1ex]
E^a_2:=&\left(F(\eps,h)-F(\eps,\ov h)\right)\{\p_x^{s-1}\kappa(\eps,\ov h)\}.
\end{align*}
Invoking \eqref{eq:cru1} (with $m=0$) and Lemma  \ref{L:5}, we get that
\begin{equation}\label{Ea1}
\begin{aligned}
\|E^a_1\|_{-1/2}\leq& C\eps^{-2}\|\p_x^{s-1}(\kappa(\eps,h)-\kappa(\eps,\ov h))
- \kappa'(\eps, h)[(h- \ov h)^{(s-1)}]\|_{1/2}\\[1ex]
\leq &C\eps^{-2}(1+\|\overline h\|_{s+3/2})\|h-\overline h\|_{s+1/2}.
\end{aligned}
\end{equation}
In order to estimate $E^a_2,$ we  write
\[
E^a_2=\int_0^1F'(\eps,rh+(1-r)\ov h)[h-\ov h]\{\p_x^{s-1}\kappa(\eps,\ov h)\}\,
dr,
\] 
and using \eqref{eq:cru1}, with $m=1$, yields
\begin{equation}\label{Ea2}
\|E^a_2\|_{-1/2}\leq \eps^{-2}\|h-\ov h\|_{s_0}\|\ov h\|_{s+3/2} \leq
C\eps^{-2}\|\ov h\|_{s+3/2}\|h-\ov h\|_{s}.
\end{equation}
Similarly, we decompose  $E^b= E^b_1+E^b_2+E^b_3,$ 
where 
\begin{align*}
E_1^b:=&\left(F'(\eps,h)-F'(\eps,\ov h)\right)[\ov
h^{(s-1)}]\{\kappa(\eps,\ov h)\}\\
=&\int_0^1F''(\eps,rh+(1-r)\ov h)[h-\ov h,\ov
h^{(s-1)}]\{\kappa(\eps,\ov h)\}\,dr,\\[1ex]
E_2^b:=&F'(\eps, h)[\ov h^{(s-1)}]\{\kappa(\eps,h)-\kappa(\eps,\ov h)\},\\[1ex]
E_3^b:=&F'(\eps,h)[h^{(s-1)}-\ov h^{(s-1)}]\{\kappa(\eps,h)\}.
\end{align*}
The estimate  \eqref{eq:cru2} with $m=1$ and $m=2$, 
respectively, yields 
\begin{align*}
\|E_1^b\|_{-1/2}&\leq C\eps^{-2}\|h-\ov h\|_{s_0}\|\ov h^{(s-1)}\|_{3/2}
\|\kappa(\eps,\ov h)\|_{s_0}\leq C \eps^{-2}\|\ov h\|_{s+1/2}\|h-\ov
h\|_{s},\\[1ex]
\|E_2^b\|_{-1/2}&\leq C\eps^{-2}\|\ov h^{(s-1)}\|_{3/2}\|
\kappa(\eps,h)-\kappa(\eps,\ov h)\|_{s_0}\leq C \eps^{-2}\|\ov
h\|_{s+1/2}\|h-\ov h\|_{s},\\[1ex]
\|E_3^b\|_{-1/2}&\leq C\eps^{-2}\|h^{(s-1)}-\ov h^{(s-1)}\|_{3/2}
\|\kappa(\eps,h)\|_{s_0}\leq C \eps^{-2}\|h-\ov h\|_{s+1/2}.
\end{align*}
To estimate  $G$, we proceed similarly and decompose 
$E^c=E^c_1+E^c_2+E^c_3$, with 
\begin{align*}
 E^c_1:=&\left(F^{(k)}(\eps,h)-F^{(k)}(\eps,\ov h)\right) [\ov
h^{(p_1)},\ldots,\ov h^{(p_k)}]\{\p_x^{p_{k+1}}\kappa(\eps,\ov h)\}\\
=&\int_0^1 F^{(k+1)}(\eps,rh+(1-r)\ov h)[h-\ov h,\ov
h^{(p_1)},\ldots,\ov h^{(p_k)}]\{\p_x^{p_{k+1}}\kappa(\eps,\ov h)\}\,dr,\\[1ex]
  E^c_2:=&F^{(k)}(\eps,h)[\ov h^{(p_1)},\ldots,\ov h^{(p_k)}]
\{\p_x^{p_{k+1}}(\kappa(\eps,h)-\kappa(\eps,\ov h))\},\\[1ex]
 E^c_3:=&\sum_{i=1}^kF^{(k)}(\eps,h)[\ov h^{(p_1)},\ldots,\ov h^{(p_{i-1})},
h^{(p_{i})}-\ov h^{(p_{i})},  h^{(p_{i+1})}, \ldots,
h^{(p_k)}]\{\p_x^{p_{k+1}}\kappa(\eps,\ov h)\}.
\end{align*}
We distinguish two cases.\\
{\em Case $1$.} Suppose first that $p_{k+1}\geq p_j$  for all $1\leq j\leq k.$
Then 
\begin{equation*}
p_{k+1}\leq s-2, \qquad p_1,\ldots, p_k\leq \frac{s-1}{2}\leq s-s_0-2, 
\end{equation*}
by the choice of $s.$
Choosing $m=k+1$, we infer from $p_{k+1}\geq 1$ that $k+1\leq s-1$,   and
together with relation \eqref{eq:cru1} we find  
 \[
\|E_1^c\|_{-1/2}\leq C\eps^{-2}\|h-\ov h\|_{s_0}\|\ov h^{(p_1)}\|_{s_0}
\ldots\|\ov h^{(p_k)}\|_{s_0}\|\ov h\|_{s+1/2}\leq C\eps^{-2}\|\ov
h\|_{s+1/2}\|h-\ov h\|_{s},
\]
while, for $m=k,$ the same relation implies
\begin{align*}
\|E^c_2\|_{-1/2}&\leq C\eps^{-2}\|h-\ov h\|_{s+1/2},\\[1ex]
\|E^c_3\|_{-1/2}&\leq C\eps^{-2}\|\ov h\|_{s+1/2}\|h-\ov h\|_{s}.
\end{align*}
{\em Case $2$.} Due to symmetry, we only have to consider the case when 
$p_1\geq
p_j,$  for all $1\leq j\leq k+1.$
Then
\begin{equation*}
p_{1}\leq s-2, \qquad p_2,\ldots, p_{k+1}\leq \frac{s-1}{2}\leq s-s_0-2, 
\end{equation*}
and \eqref{eq:cru2} with $m=k+1$ and $m=k,$ respectively, yields
\[
\|E_i^c\|_{-1/2}\leq C\eps^{-2}\|h-\ov h\|_{s}.
\]
This completes the proof.
\end{proof}

\section{Approximation by power series in $\eps$}\label{sec4}

In this section we construct  operators $\clf_k$ and functions $t\mapsto
h_{\eps,k}(t)$ such that, in a sense to be made precise below, 
\[\clf(\eps, h)=\clf_k(\eps,h)+O(\eps^{k+1}),\]
and $h_{\eps,k}$ is an approximate solution to \eqref{eq:CP}. Formally, the
construction is by expansion with respect to $\eps$ near $0$, i.e.,
$\clf_k(\eps,h)$ and $h_{\eps,k}$ are polynomials of order $k$ in $\eps$. In
lowest order $k=0$, we will recover the Thin Film equation \eqref{thf}.
As this construction involves a loss of regularity that increases with $k$,
we will have to assume higher smoothness of $h$.

Fix $k\in\NNM$ and let $s_1\geq s+k+15/2$, with $s$ as before. In this section,
we will assume $h\in\clu_{s_1}$ and all constants in our estimates will be
independent of $h$. 

We start with a series expansion for $w(\eps,h)\{f\}$. 
\begin{lemma} For $p=0,1,\ldots,k+2$ there are operators
\[w^{[p]}\in C^\infty\left(\clu_{s_1},
\cll\left(H^{s_1-2}(\SSM),H^{s_1-2-p}(\0)\right)\right)\]
such that
\[\left\|w(\eps,h)\{f\}-\sum_{p=0}^{k+2}\eps^p
w^{[p]}(h)\{f\}\right\|_{s+3/2}^{\0}
\leq C\eps^{k+3}\|f\|_{s_1-2}.\]
In particular,
\[w^{[0]}(h)\{f\}(x,y)=f(x),\quad
w^{[2]}(h)\{f\}(x,y)=f''(x)\int_y^1\frac{\tau}{a_2^2(x,\tau)}\,d\tau,\quad
w^{[p]}=0\mbox{ for $p$ odd.}\]
\end{lemma}
\begin{proof}
 Recalling \eqref{repA} we have $\cla(\eps,h)=\cls_0(h)+\eps^2\cls_2(h)$ with
\begin{equation}\label{eq:cls}
\begin{aligned}
\cls_0(h):=&a_2^2\p_{22}+a_2a_{2,2}\p_2,\\[1ex]
\cls_2(h):=&\p^2_{11}+2a_1\p^2_{12}+a_1^2\p_{22}^2+\left(a_{1,1}+a_1a_{1,2}
\right)\p_2.
\end{aligned}
\end{equation}
The terms $w^{[p]}(h)\{f\}$ are determined successively from inserting the
ansatz 
\[w(\eps,h)\{f\}=\sum_{p=0}^{k+2}\eps^p w^{[p]}(h)\{f\}+R\]
into 
\[\left\{\begin{array}{rcll}
(S_0(h)+\eps^2S_2(h))w(\eps,h)\{f\}&=&0&\mbox{in $\Omega$,}\\
w(\eps,h)\{f\}&=&f\circ\phi_\pm&\mbox{on $\Gamma_\pm$,}
\end{array}\right.\]
and equating terms with equal powers of $\eps$. Thus we obtain
\[
\left\{
\begin{array}{rlllll}
\cls_0(h)w^{[0]}&=&0& \text{in}&\Omega,\\[1ex]
w^{[0]}&=&f\circ\phi_\pm&\text{on}& \Gamma_\pm
\end{array}
\right.\qquad
\left\{
\begin{array}{rlllll}
\cls_0(h)w^{[1]}&=&0& \text{in}&\Omega,\\[1ex]
w^{[1]}&=&0&\text{on}& \p\Omega.
\end{array}
\right.
\]
and further 
\[
\left\{
\begin{array}{rlllll}
\cls_0(h)w^{[p+2]}&=&-\cls_2(h) w^{[p]}& \text{in}&\Omega,\\[1ex]
w^{[p+2]}&=&0&\text{on}& \p\Omega,
\end{array}
\right.
\]
$p=0,\ldots,k$. 
Observe that the general problem
\[
\left\{
\begin{array}{rlllll}
\cls_0(h)u&=&G& \text{in}&\Omega,\\[1ex]
u&=&g\circ\phi_\pm&\text{on}& \Gamma_\pm
\end{array}
\right.\]
(with $g$ and $G$ even) is solved by
\[u(x,y)=g(x)-\int_y^1\frac{1}{a_2(x,\tau)^2}\int_0^\tau G(x,s)\,ds\,d\tau,\]
and for this solution we have
\[\|u\|_t^{\0}\leq C\left(\|g\|_t+\|G\|_t^{\0}\right),\]
$t\in[s+3/2,s_1-2]$. 
All statements concerning the mapping properties and the explicit form of the
$w^{[p]}$ follow from this. To estimate the remainder, observe that
\[
\left\{
\begin{array}{rlllll}
\cla(\eps,h)R&=&-\eps^{k+3}\cls_2(h) w^{[k+1]}(h)\{f\}- \eps^{k+4}\cls_2(h)
w^{[k+2]}(h)\{f\}& \text{in}&\Omega,\\[1ex]
R&=&0&\text{on}& \p\Omega.
\end{array}
\right.
\]
The estimate follows from Proposition \ref{P:1} with $s$ replaced by $s_1$ and
$t=s+5/2$. 
\end{proof}
Recall, furthermore, that
\[\clb(\eps,h)=\eps^{-2}\clb^{[0]}(h)+\clb^{[2]}(h)\]
where
\[\clb^{[0]},\clb^{[2]}\in
C^\infty\left(\clu_{s_1},\cll\left(H^{s+3/2}(\0),H^s(\SSM)\right)\right)\]
are given by
\[\clb^{[0]}(h)w=h^{-1}(\partial_2w)|_{\Gamma_+}\circ\phi_+^{-1},\quad
\clb^{[2]}(h)w=-h'(\p_1w)|_{\Gamma_+}\circ\phi_+^{-1}+h^{-1}
(h')^2(\partial_2w)|_{\Gamma_+}\circ\phi_+^{-1}.\]
By Taylor expansion around $\eps=0$ it is straightforward to see that there are
functions
\[\kappa^{[p]}\in C^\infty\left(\clu_{s_1}, H^{s_1-2}(\SSM)\right),\qquad
p=0,\ldots,k+2,\]
such that
\[\left\|\kappa(\eps,h)-\sum_{p=0}^{k+2}\eps^p\kappa^{[p]}(h)\right\|_{s_1-2}
\leq C\eps^{k+3}.\]
In particular,
\[\kappa^{[0]}(h)=h''\quad\text{and } \kappa^{[p]}=0\text{ for $p$ odd}.\]
In view of \eqref{defF}, \eqref{defclf} we define
\[\clf_k(\eps,h):=\sum_{p=0}^{k+2}\eps^{p-2}
\sum_{j+m+l=p}\clb^{[j]}(h)w^{[m]}(h)\{\kappa^{[l]}(h)\},\]
$j\in\{0,2\}$. As all terms corresponding to $p=0$ and $p=1$ vanish, this is
indeed a polynomial in $\eps$ and
\[\clf_k\in C^\infty\left([0,1)\times\clu_{s_1}, H^s(\SSM)\right).\]
In particular,
\[\clf_0(\eps,h)=\clf_k(0,h)=\clb^{[0]}(h)w^{[2]}(h)\{\kappa^{[0]}(h)\}+
\clb^{[2]}(h)w^{[0]}(h)\{\kappa^{[0]}(h)\}=-(hh''')'\]
(cf. \eqref{thf}).

It is straightforward now to obtain
\begin{equation}\label{eq:ESS}
\|\clf(\eps,h)-\clf_k(\eps,h)\|_{s}\leq C\eps^{k+1}.
\end{equation}

To construct the approximation $h_{\eps,k}$ we start with an arbitrary,
sufficiently smooth, strictly positive solution $h_0$ of the Thin Film equation
\eqref{thf} and successively add higher order corrections. We closely follow 
\cite[Lemma 5.3]{guprth} here. Fix $T>0$, $h^\ast\in\clu_s$, and set for brevity
\begin{align*}
\tau&:=k+15/2,\\
s_2&:=s_2(k,s):=s+[k/2](\tau-4)+\tau+1,\\
\clv_\sigma&:=\big\{H\in C\big([0,T],\clu_s\cap H^\sigma(\SSM)\big)
\cap C^1\big([0,T],
H^{\sigma-4}(\SSM)\big)\,\big|\,H(0)=h^\ast\big\},\quad\sigma\geq s.
\end{align*}
Let $h_0\in\clv_{s_2}$ 
be a  solution to \eqref{thf}. Observe that
\be\label{smFke}
\clf_k\in C^{\infty}\left([0,1)\times\clu_\sigma,
H^{\sigma-\tau}(\SSM)\right),\qquad\sigma\in[s+\tau,s_2].
\ee
Furthermore, for $t\in[0,T]$, the linear fourth order differential operator
\[A:=A(t):=\partial_h\clf_k(0,h_0(t))=[\,h\mapsto(hh_0'''(t)+h_0(t)h''')'\,]\]
is elliptic, uniformly in $x$ and $t$. 

\begin{lemma}\label{approxhek} Fix $h_0$ as above. There are positive constants
$\eps_0$ and $C$
and functions \linebreak
$h_{\eps,k}\in\clv_{s+4}$, $\eps\in[0,\eps_0)$,
that satisfy 
\be\label{asyhe}
\int_\SSM h_{\eps,k}(t)\,dx 
=\int_\SSM h_0(0)\,dx\quad
\text{and}\quad
\|\partial_th_{\eps,k}(t)-\clf(\eps,h_{\eps,k}(t))\|_s\leq
C\eps^{k+1},\quad t\in[0,T].
\ee
\end{lemma}
\begin{proof}
We construct $h_{\eps,k}$ by the ansatz
\[h_{\eps,k}:=h_0+\eps h_1+\ldots+\eps^kh_k,\]
where for $p=1,\ldots,k$, $h_p$ is recursively determined from
$h_0,\ldots,h_{p-1}$ as solution of the fourth order linear parabolic Cauchy
problem 
\[\left\{\begin{array}{rcl}
\partial_th_p&=&
\frac{1}{p!}\frac{d^p}{d\eps^p}(\clf_k(\eps,h_{\eps,k})|_{\eps=0}
=Ah_p+R_p,\\
h_p(0)&=&0,
\end{array}\right.\]
where $R_p=[t\mapsto R_p(t)]$ is a finite sum of terms of the form 
\[\partial_{\eps}^l\partial_h^m\clf_k(0,h_0)[h_{j_1},\ldots,h_{j_m}],
\qquad 1\leq j_i \leq p-1,\quad l+\sum_{i=1}^m j_i=p.\]
(At this point, the expression
$\frac{d^p}{d\eps^p}(\clf_k(\eps,h_{\eps,k})|_{\eps=0}$ should be understood in
the sense of formal expansions only. It will be justified below.)
Note that $\partial_{\eps}^l\clf_k(0,h_0)=0$ for $l$ odd, and therefore also
$h_p=0$ for $p$ odd. Fix $\theta\in(0,1/4)$. We will show by induction that 
\[h_p\in C^\theta([0,T],H^{\sigma_p}(\SSM))\cap
C^{1+\theta}([0,T],H^{\sigma_p-4}(\SSM)),\qquad
\sigma_p:=s_2-1-\frac{p}{2}(\tau-4).\]
For $p=0$, this follow from our assumptions by a standard interpolation
argument.  Suppose now this is true up to some even
$p\leq k-2$. By \eqref{smFke} we find
\[R_{p+2}\in C^\theta([0,T], H^{\sigma_p-\tau}(\SSM)),\quad R_p(0)\in
H^{s_2-\tau}(\SSM)\]
and by standard results on linear parabolic equations 
(cf. e.g. \cite[Prop.6.1.3]{lun})
\[h_{p+2}\in C^\theta([0,T],H^{\sigma_p-\tau+4}(\SSM))\cap
C^{1+\theta}([0,T],H^{\sigma_p-\tau}(\SSM)).\]
Therefore, by our choice of $s_2$,
\[h_{\eps,k}\in C([0,T],H^{s+\tau}(\SSM))\cap
C^1([0,T],H^{s+\tau-4}(\SSM)).\]
If $\eps\in[0,\eps_0)$, $\eps_0$ sufficiently small, this implies
$h_{\eps,k}(t)\in\clu_{s+\tau}$, and thus,
 by Taylor's theorem applied at $\eps=0$ to $\eps\mapsto\partial_t
h_{\eps,k}-\clf_k(\eps,h_{\eps,k})$,
\[ \|\partial_t h_{\eps,k}-\clf_k(\eps,h_{\eps,k})\|_s\leq C\eps^{k+1}\]
Consequently, we get \eqref{asyhe} from this and \eqref{eq:ESS}.

Finally, for all $h\in\clu_{s+\tau}$ we have $\int_\SSM\clf(\eps,h)\,dx=0$, cf.
\cite[Lemma 3.1]{EM1} and \cite[Lemma 1]{prhs}.
Therefore, by \eqref{asyhe}, $\int_\SSM\partial_t h_{\eps,k}\,dx=O(\eps^{k+1})$.
This implies $\int_\SSM h_p\,dx=0$, $p=1,\ldots,k$, and thus the lemma is
proved completely.
\end{proof}

\section{Proof of the main result}\label{sec5}

Let $T$, $\clu_s$, $h^\ast$ as in the previous section
and fix $T'\in(0,T]$. Let 
\[h_\eps\in C\big([0,T'],\clu_s\big)\cap C^1\big([0,T'],H^{s-3}(\SSM)\big)\] be
a solution of
\eqref{eq:CP} with $h_\eps(0)=h^\ast$. 
For given, sufficiently smooth $h_0$ solving \eqref{thf}, we denote by
$h_{\eps,k}$ the function constructed in Lemma \ref{approxhek}.
The following energy estimates
are the core of our result. 
\begin{prop}\label{pr:apri} \mbox{ }
\begin{itemize}
\item[\rm (i)] Fix $k\in\NNM$ and a solution $h_0\in\clv_{s_2(k)}$ of
\eqref{thf}. There are  constants $C$ and $\eps_0$ depending on $\clu_s,k,T$,
and $h_0$ such that
\be\label{h1est}
\|h_\eps(t)-h_{\eps,k}(t)\|_1\leq
C\eps^{k+1},\qquad\eps\in(0,\eps_0),t\in[0,T'].
\ee
\item[\rm (ii)] Fix $n\in\NNM$. There is a $\beta=\beta(s,n)\in\NNM$ such that
for any solution $h_0\in\clv_{\beta}$ to \eqref{thf} there are  constants $C$
and $\eps_0$ depending on $\clu_s,n,T$,
and $h_0$ such that
\be\label{hsest}
\|h_\eps(t)-h_{\eps,n-1}(t)\|_s\leq
C\eps^n,\qquad\eps\in(0,\eps_0),t\in[0,T'].
\ee
\end{itemize}
\end{prop}
\begin{proof} (i) Let $\eps_0$ be small enough to ensure that
$h_{\eps,k}(t)\in\clu_s$, $\eps\in[0,\eps_0)$, $t\in[0,T]$. 

 We introduce the differences 
\begin{equation*}
d(t):=h_\eps(t)-h_{\eps,k}(t)\qquad\text{and} \qquad
\delta(t):=\kappa(\eps,h_\eps(t))-\kappa(\eps,h_{\eps,k}(t)).
\end{equation*}
 We obviously have
\be\label{repdel}
\delta(t)=\int_0^1\kappa'(\eps,\tau h_{\eps}(t)-(1-\tau)h_{\eps,k}(t))[d(t)]\,
d\tau, \qquad t\in[0, T'],
\ee 
and since 
\[
\kappa'(\eps,h)[d]=\left(\frac{d'}{(1+\eps^2h'^2)^{3/2}}\right)',
\]
we obtain that there exist positive constants $c_{1,2}=c_{1,2}(\clu_s)$ such
that 
\be\label{equivd}
c_1\|d\|_{\sigma}\leq\|\delta\|_{\sigma-2}\leq
c_2\|d\|_{\sigma},\qquad\sigma\in[1,2].
\ee
(Here and in the sequel, we will omit the argument $t$  if no
confusion is likely.)
In the same spirit, for $h,\bar h\in\clu_s$, we introduce the bilinear form
$B(\eps,h,\bar h):H^1(\SSM)\times H^1(\SSM)\to\RRM$
by
\[
B(\eps,h,\bar h)(e,f):=\int_0^1  \int_\SSM \frac{e'f'}{(1+\eps^2(\tau 
h'+(1-\tau)\bar h')^2)^{3/2}}  \, d\sigma\, d\tau .
\]
Observe that there are positive constants $c_{1,2}=c_{1,2}(\clu_s)$ such
that 
\be\label{equivip}
c_1\|d\|_1^2\leq B(\eps, h, \bar h)(d,d)\leq c_2 \|d\|_1^2
\ee
as $d(t)$ has zero average over $\SSM$.

From \eqref{repdel} and \eqref{eq:Boun} we find, via integration by parts,
\begin{align}\label{curvest}
-\langle \p_t d\,|\,\delta\rangle_{L^2(\SSM)}&=B(\eps,
h_\eps,h_{\eps,k})(d,\p_td)\nonumber\\
&=\tfrac{1}{2}\Big((\partial_t\big(B(\eps,h_\eps,h_{\eps,k})(d,
d)\big)\nonumber\\
&\phantom{=}-\partial_hB(\eps , h_\eps,h_{\eps,k})(d,d)\partial_t
h_\eps-\partial_{\bar h}B(\eps,
h_\eps,h_{\eps,k})(d,d)\partial_t h_{\eps,k}\Big)\nonumber\\
&\geq\tfrac{1}{2}\partial_t\big(B(\eps,h_\eps,h_{\eps,k})(d,
d)\big)-C\|d\|_1^2.
\end{align}

Furthermore, from \eqref{eq:CP} and \eqref{asyhe} we have
\begin{align}\label{evold}
\p_t d(t)=&F(\eps,h_\eps(t))\{\kappa(\eps,h_\eps(t))\}
-F(\eps,h_{\eps,k}(t))\{\kappa(\eps,h_{\eps,k}(t))\}+R(t)\nonumber\\[1ex]
=&F(\eps,h_\eps(t))\{\delta(t)\}+\wt R(t)+R(t),
\end{align}
where
\begin{equation}\label{eq:qe1}
\max_{[0,T]}\|R(t)\|_s\leq C\eps^{k+1}
\end{equation}
and
\[
\wt R:=\int_0^1F'(\eps, \tau
h_{\eps,k}+(1-\tau)h_\eps)[d]\{\kappa(\eps,h_{\eps,k})\}\, d\tau
\]
By Lemma \ref{L:32} and \eqref{equivd}, 
\begin{equation}\label{eq:qe2}
\|\wt R\|_{1/2}\leq C\|d\|_{3/2}
\|\kappa(\eps,h_{\eps,k})\|_{s_0+3/2}\leq C\|\delta\|_{-1/2},\qquad
 \eps\in (0,\eps_0),\;t\in [0,T'].
\end{equation}
Multiplying \eqref{evold} by $-\delta$ and applying \eqref{eq:ana},
\eqref{eq:qe1},  \eqref{eq:qe2},  and an interpolation inequality we get
\begin{align*}
-\langle \p_t d\,|\,\delta\rangle_{L^2(\SSM)}&\leq -c\|\delta\|_{1/2,\eps}^2+
C\|\delta\|_{-1/2}^2+C\eps^{k+1}\|\delta\|_{-1}\\[1ex]
&\leq -c\|\delta\|_{0}^2+\left(c\|\delta\|_{0}^2+C\|\delta\|_{-1}^2\right)
+C\left(\eps^{2k+2}+\|\delta\|_{-1}^2\right)\\[1ex]
&\leq C\left(\|\delta\|_{-1}^2+\eps^{2k+2}\right).
\end{align*}
Together with \eqref{curvest}, this shows that
\[
\frac{d}{dt} B(\eps, h_\eps,h_{\eps,k})(d,d)\leq C(\eps^{2k+2}+ B(\eps,
h_\eps,h_{\eps,k})(d,d)) 
\]
for all $\eps\in (0,\eps_0)$, $t\in [0,T'].$
Taking into consideration that $d(0)=0,$ we find by Gronwall's inequality
that
\[
c_1\|d\|_1^2\leq B(\eps, h_\eps,h_{\eps,k})(d,d)\leq C(T)\eps^{2k+2},
\]
which proves \eqref{h1est}.

(ii) Set $k:=n+5s-1$ and $\beta:=s_2(k)$. Let $\eps_0$ be small enough to
ensure that
$h_{\eps,k}(t)\in\clu_s$, $\eps\in[0,\eps_0)$, $t\in[0,T]$.

Instead of \eqref{hsest} we are going to prove the equivalent estimate
\be\label{hsest2}
\|h_\eps(t)-h_{\eps,k}(t)\|_s\leq
C\eps^n,\qquad\eps\in(0,\eps_0),t\in[0,T'].
\ee

 Let $\delta_{s-1}:=\kappa'(\eps,h_\eps)[d^{(s-1)}]$. Then, in analogy to 
\eqref{curvest},
\be\label{curvest2}
-\langle \p_td^{(s-1)}\,|\,\delta_{s-1}\rangle_{L^2(\SSM)}
\geq\tfrac{1}{2}\p_t(B(\eps,h_\eps)(d^{(s-1)},d^{(s-1)}))-C\|d\|_s^2,
\ee
where $B(\eps,h_\eps):=B(\eps,h_\eps,h_\eps)$. 

On the other hand, differentiating the relation 
\[
\p_td=\clf(\eps,h_\eps)-\clf(\eps,h_{\eps,k})+R
\]
$(s-1)$ times with respect to $x$ we get (cf. Lemma \ref{L:6})
\[\p_td^{(s-1)}=-F(\eps,h_\eps)[\delta_{s-1}]+P_s(\eps,h_\eps,h_{\eps,k})+
R^{(s-1)}.\]
Recalling \eqref{eq:ana}, \eqref{eq:L6}, \eqref{equivd}, and \eqref{eq:qe1} we
obtain from this by Young's inequality
\begin{align*}
-\langle \p_td^{(s-1)}\,|\,\delta_{s-1}\rangle_{L^2(\SSM)}
&\leq -c\|\delta_{s-1}\|^2_{1/2,\eps}
+C\eps^{-2}\|d\|_{s+1/2}\|\delta_{s-1}\|_{1/2}
+C\eps^{k+1}\|\delta_{s-1}\|_{-1}\\
&\leq -c\|\delta_{s-1}\|^2_0+C\eps^{-5}\|d\|^2_{s+1/2}+C\eps^{2k+2}
+\|\delta_{s-1}\|^2_{-1}\\
&\leq -c\|d\|_{s+1}^2+C\eps^{-5}\|d\|^2_{s+1/2}+C\eps^{2k+2}.
\end{align*}
Consequently, by \eqref{curvest2}, \eqref{h1est}, and an interpolation
inequality,
\begin{align*}
\p_tB(\eps,h_\eps)(d^{(s-1)},d^{(s-1)})&\leq -c|d\|_{s+1}^2
+C\eps^{-5}\|d\|_{s+1}^{2-1/s}\|d\|_1^{1/s}+C\eps^{2k+2}\\
&\leq C(\eps^{-10s}\|d\|_1^2+\eps^{2k+2})\leq C\eps^{2k+2-10s}
\leq C\eps^{2n}.
\end{align*}
Integrating over $t$ and using  \eqref{equivip}, we obtain
\eqref{hsest2}.
\end{proof}
\begin{proof}[Proof of Theorem {\rm\ref{thm2}}] Choose $k$ and
$\beta=\beta(s,n)$ as in the proof of Proposition \ref{pr:apri} (ii), let
$h^*:=h_0(0)$ and let $\alpha$ and $M$ be such that
$h_0([0,T])\subset\clu_s$. 
By compactness, $\mu:=\mathop{\rm dist}(\p\clu_s,h_0([0,T]))>0$. Let $\eps_0$ be
small enough to ensure that $h_{\eps_,k}([0,T])\subset\clu_s$, 
\be\label{distest}
\mathop{\rm dist}(\p\clu_s,h_{\eps,n-1}([0,T]))>\mu/2,\qquad\eps\in[0,\eps_0),
\ee
and $C\eps_0^n<\mu/4$, where $C$ is the constant from \eqref{hsest}. 

Let $\eps\in(0,\eps_0)$ and let
\[h_\eps\in C\big([0,T_\eps),H^{\beta-1}(\SSM)\big)\cap
C^1\big([0,T_\eps),H^{\beta-4}(\SSM)\big)\]
be a maximal solution to \eqref{eq:CP} with $h_\eps(0)=h^\ast$.
In view of Proposition \ref{pr:apri} (ii), it remains to show that
$T_\eps>T$. Assume $T_\eps\leq T$. The blowup result in Theorem \ref{thm1} (iii)
implies that there is a $T'\in(0,T)$ such that
$h_\eps([0,T'])\subset\clu_s$ but $\mathop{\rm dist}(\p\clu_s,h(T'))<\mu/4.$
In view of \eqref{hsest} and \eqref{distest}, this is a contradiction to our
choice of $\eps_0$. 
\end{proof}

\subsubsection*{Acknowledgements}
The research leading to this paper was carried out in part while the second
author enjoyed the hospitality of the Institute of Applied Mathematics as a
guest researcher at Leibniz University Hannover. Moreover, we express our
gratitude to Prof. M. G\"unther (Leipzig University) whose ideas for
\cite{guprth} are crucial for the present paper as well.

\begin{small}

\end{small}
\end{document}